\documentclass[11pt]{article}
\usepackage[a4paper, margin=1in]{geometry}

\usepackage{amsmath,amssymb,amsfonts,amsthm}
\allowdisplaybreaks[4]

\usepackage{graphicx}
\usepackage{pdflscape}

\usepackage{array}
\usepackage{booktabs}
\usepackage{multirow}
\usepackage{longtable}
\usepackage{supertabular} 
\usepackage{adjustbox}

\usepackage{algorithm}
\usepackage{algorithmic}

\usepackage{listings}
\usepackage{enumitem}
\usepackage{verbatim}
\usepackage{placeins}
\usepackage{cases}
\usepackage{empheq}
\usepackage{tikz}
\usepackage{textcomp}
\usepackage{stfloats} 

\usepackage{cite}

\definecolor{dred}{rgb}{0.8,0.2,0.0}

\usepackage{url}
\usepackage[colorlinks=true, linkcolor=blue, citecolor=blue, urlcolor=blue]{hyperref}
\usepackage{orcidlink}
\usepackage{balance}

\newtheorem{theorem}{Theorem}
\newtheorem{assumption}{Assumption}

\newtheorem{remark}{Remark}


\begin{document}
\title{
Graph-Guided Fused Regularization for Single- and Multi-Task  Regression on Spatiotemporal Data\\
}

\author{Meixia Lin\thanks{Engineering Systems and Design, Singapore University of Technology and Design, Singapore (e-mail: meixia\_lin@sutd.edu.sg).} \and
        Ziyang Zeng\thanks{Department of Industrial Systems Engineering and Management, National University of Singapore (e-mail: ziyangzeng@u.nus.edu).} \and
         Yangjing Zhang\thanks{The State Key Laboratory of Mathematical Sciences, Academy of Mathematics and Systems Science, Chinese Academy of Sciences, Beijing, China (e-mail: yangjing.zhang@amss.ac.cn).}% 
        }
\maketitle

\begin{abstract}
Spatiotemporal matrix-valued data arise frequently in modern applications, yet performing effective regression analysis remains challenging due to complex, dimension-specific dependencies. In this work, we propose a regularized framework for spatiotemporal matrix regression that characterizes temporal and spatial dependencies through tailored penalties. Specifically, the model incorporates a fused penalty to capture smooth temporal evolution and a graph-guided penalty to promote spatial similarity. The framework also extends to the multi-task setting, enabling joint estimation across related tasks. We provide a comprehensive analysis of the framework from both theoretical and computational perspectives. Theoretically, we establish the statistical consistency of the proposed estimators. Computationally, we develop an efficient solver based on the Halpern Peaceman-Rachford method for the resulting composite convex optimization problem. The proposed algorithm achieves a fast global non-ergodic $\mathcal{O}(1/k)$ convergence rate with low per-iteration  complexity. Extensive numerical experiments demonstrate that our method significantly outperforms state-of-the-art approaches in terms of predictive accuracy and estimation error, while also exhibiting superior computational efficiency and scalability.

\end{abstract}

% \begin{IEEEkeywords}
% Graph-guided regularization, 
% Halpern Peaceman–Rachford method,
% multi-task learning,
% spatiotemporal
% regression, 
% structured sparsity.
% \end{IEEEkeywords}

\section{Introduction}
\label{sec: intro}
We consider the matrix regression
model
\begin{equation}\label{eq: model, single}
y_k = \langle X_k, \theta \rangle + \varepsilon_k, \quad k \in [n]:=\{ 1,\dots,n\},
\end{equation}
where $n$ is the sample size, $y_k \in \mathbb{R}$ is the scalar response, and $\varepsilon_k$ denotes the observation noise. {The unknown regression coefficient $\theta \in \mathbb{R}^{t \times s}$ and the predictors $X_k \in \mathbb{R}^{t \times s}$ are matrices that encode spatiotemporal structures},
with rows and columns representing $t$ time lags and $s$ spatial locations, respectively, 
as commonly encountered in applications such as climate modeling (e.g., sea surface temperature data).

One line of existing approaches relies on sparsity-inducing penalties like the lasso \cite{Tibshirani1996}, structured sparsity regularizations including the fused lasso \cite{Tibshirani2005} and group lasso \cite{yuan2006model}, as well as low-rank regularization \cite{chen2023quantized}. However, these methods may fail to fully exploit the intrinsic spatiotemporal dependencies {in} the predictors.
To address this,
another line of work assumes that the regression coefficient aligns well
with the correlation structure of the predictors, and employs graph total variation
(GTV) regularization to promote similarity among highly correlated variables (see, e.g.,
\cite{li2018graph,stevens2019graph,li2020graph,stevens2021graph}). These approaches typically treat time lags and spatial locations jointly and rely heavily on estimating a $ts\times ts$ covariance (cf. \eqref{eq: gtv}). {In practice, an} accurate estimation of such a covariance can be statistically challenging and often requires substantial side information. Moreover, such joint modeling may overlook dimension-specific characteristics {inherent} in spatiotemporal data, where temporal evolution is smooth and adjacent locations exhibit similar behavior (see Fig. \ref{fig:example1}).

{{Motivated by these considerations, w}e propose a {regularized}
framework that treats time lags and spatial locations separately, encouraging smooth temporal evolution and local spatial similarity. {Specifically,}
we impose a fused penalty along the temporal dimension to promote smoothness across time lags, and a graph-guided regularization across spatial locations to encourage similar behaviors among adjacent regions. Unlike existing GTV-based approaches that rely on estimating a $ts \times ts$ covariance matrix, our method only {uses a prespecified}
$s \times s$ spatial graph,
constructed directly from geographic proximity or 
{simply a}
lattice {structure}
when side information is unavailable.}

\begin{figure}[!t]
  \centering
    \includegraphics[width=0.33\columnwidth]{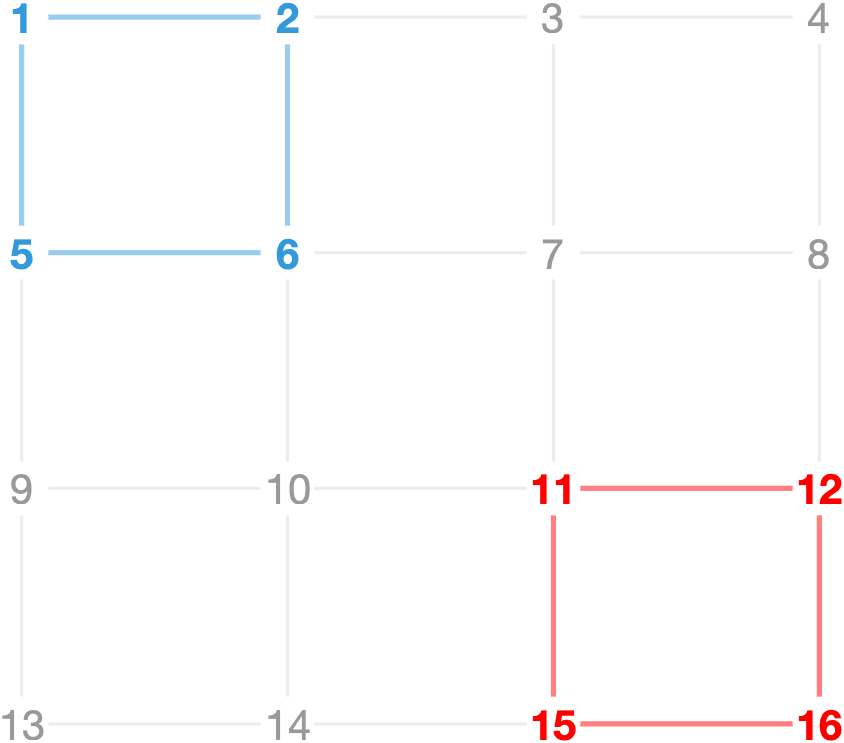}
  \hfil
  \includegraphics[width=0.35\columnwidth]{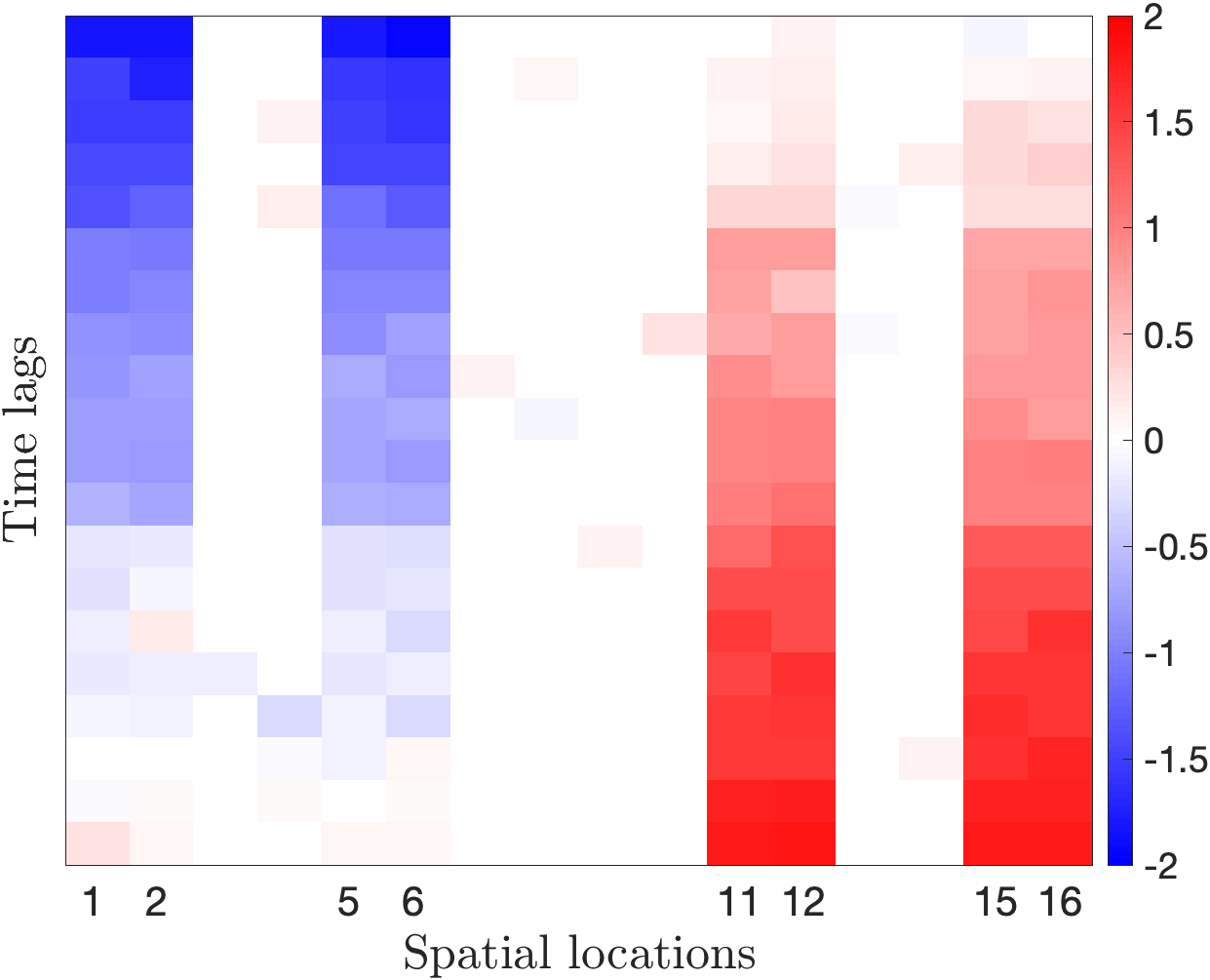}

  \caption{Illustrative spatiotemporal structure in the coefficient matrix. Left: Spatial graph on a $4\times4$ grid,
  where blue and red groups highlight 
  sets of adjacent {locations}.
  Right: Estimated coefficient matrix $\theta$, exhibiting smooth  evolution over time lags, and similarity among adjacent spatial locations.
  }
  \label{fig:example1}
\end{figure}

We further consider a multi-task extension
\begin{equation}\label{eq: model, multi}
y_k^{(r)} = \langle X_k, \theta^{(r)} \rangle + \varepsilon_k^{(r)}, \quad k \in[n],\  r \in [m],
\end{equation}
where each task $r$ has its own regression coefficient $\theta^{(r)} \in \mathbb{R}^{t \times s}$, response $y_k^{(r)}\in \mathbb{R}$ and noise $\varepsilon_k^{(r)}$, while the predictors $X_k \in \mathbb{R}^{t \times s}$ are shared across tasks. To capture relationships among tasks, we adopt a common multi-task learning strategy based on joint $\ell_2$-type regularization, which encourages similar support patterns and information sharing across tasks; see, e.g., \cite[(3)]{stevens2019graph}, \cite[(3)]{chen2010graph}.

{Altogether, we integrate the aforementioned components into}
a unified framework for single- and multi-task {spatiotemporal} regression{, which} 
is formulated as a composite convex optimization problem that combines
a sparsity-inducing $\ell_1$ penalty, a fused temporal penalty,
a graph-guided spatial penalty,
and, in the multi-task setting, an additional $\ell_2$ penalty. 
A standard method for such problems is the alternating direction method of multipliers (ADMM) \cite{gabay1976dual, glowinski1975approximation, eckstein1992douglas}. However, ADMM is known to achieve only non-ergodic rates of {$o(1/\sqrt{k})$} for primal feasibility and objective {gap} \cite{davis2017convergence}, and $\mathcal{O}(1/\sqrt{k})$ for Karush-Kuhn-Tucker (KKT) {residual} \cite{cui2016convergence}, motivating the development of accelerated variants. A recent and effective acceleration strategy incorporates {the} Halpern’s iteration into classical operator splitting schemes. In particular, Zhang et al.~\cite{zhang2022efficient} proposed the Halpern Peaceman–Rachford (HPR) method and established a non-ergodic $\mathcal{O}(1/k)$ convergence rate for both the KKT residual and the primal objective gap. 
Subsequent work showed that HPR is equivalent to a Halpern-accelerated ADMM \cite{chen2024hpr}, and similar acceleration guarantees were obtained for proximal ADMM schemes \cite{sun2025accelerating}.
Beyond  theoretical improvements, HPR has shown strong practical performance in linear programming \cite{chen2024hpr}, convex quadratic programming \cite{chen2025hpr}, and optimal transport \cite{zhang2025hot}. 
These developments motivate the adoption of an HPR-based algorithm for efficiently computing our proposed  estimator.

Our main {contributions are summarized as follows.
\begin{itemize}
\item We propose a unified {regularized} framework for single- and multi-task spatiotemporal matrix regression that separately models temporal smoothness and spatial {similarities},
and we {further} establish the consistency of the proposed estimator.
\item Unlike existing GTV-based approaches \cite{li2018graph,stevens2019graph,li2020graph,stevens2021graph}, the proposed method avoids estimating high-dimensional   $ts\times ts$  covariance matrices and only 
{{uses a prespecified}} 
$ s\times  s$  spatial graph, which can be constructed directly from geographic proximity or a simple lattice structure.
\item We develop an efficient HPR-based algorithm for the resulting composite convex optimization problem with provable non-ergodic $\mathcal{O}(1/k)$ convergence guarantees, and demonstrate its practical advantages through extensive numerical experiments.
\end{itemize}}

The rest of the paper is organized as follows. Section~\ref{sec: model} introduces the proposed single- and multi-task estimators. Section~\ref{sec: consistency} establishes their consistency properties. Section~\ref{sec: alg} develops an efficient HPR-based algorithm for solving the resulting regularized matrix regression problems. Extensive numerical experiments are presented in Section~\ref{sec: experiment}, and we conclude in Section~\ref{sec: conclusion}.

\smallskip
\noindent \textbf{Notations.}
Let $[n]:=\{1,\dots,n\}$ and $I_s$ be the $s\times s$ identity matrix {(or simply $I$ when the dimension is clear from the context)}. 
The symbol $\otimes$ denotes the Kronecker product. 
For a matrix $A$, $\operatorname{sgn}(A)$ is applied componentwise, where {for $t\in\mathbb{R}$,} $\operatorname{sgn}(t)=1$ if $t>0$, $\operatorname{sgn}(t)=0$ if $t=0$, and $\operatorname{sgn}(t)=-1$ if $t<0$. 
For a set $C$, the {characteristic function }$\mathbb{I}_C$ is defined as $\mathbb{I}_C(x)=1$ if $x\in C$ and $\mathbb{I}_C(x)=0$ otherwise.
For a vector $x$, $\|x\|_q$ denotes its $\ell_q$ norm. 
For a matrix $X$, $X_{i\cdot}$ and $X_{\cdot j}$ denote its $i$-th row and $j$-th column, respectively. 
For a collection of matrices $\{A^{(r)}\}_{r=1}^m $, we define its Frobenius norm as $
\|\{A^{(r)}\}_{r=1}^m\|_F
:=  ( \sum_{r=1}^m \|A^{(r)}\|_F^2  )^{1/2}.
$
Let $h:\mathbb{R}^n\to(-\infty,\infty]$ be a closed proper convex function, the proximal mapping of $h$ at $x$ is defined as
$
\operatorname{prox}_h(x) := \arg\min_{y\in\mathbb{R}^n}
\left\{ h(y) + \tfrac{1}{2}\|y-x\|^2 \right\}.
$

\section{Model Formulation}
\label{sec: model}

In this section, we formally present our regression frameworks for both {single-} and multi-task settings. Building upon the motivations in Section \ref{sec: intro}, we design estimators that simultaneously capture sparsity, temporal smoothness, and spatial {similarities.} 
We first formalize the single-task model in Section~\ref{sec:singletask}, then extend it to the multi-task setting in Section~\ref{sec:multitask}, where 
{the} coefficient matrices {of related tasks} are estimated jointly to promote information and sparsity sharing.

\subsection{GGFL: Single-Task  Model}\label{sec:singletask}
Given observations $\{(X_k,y_k) \mid k \in [n]\}$ from model~\eqref{eq: model, single}, we define the linear operator $\mathcal{X} : \mathbb{R}^{t \times s} \to \mathbb{R}^n$ by 
$
\mathcal{X}(\theta) = [\langle X_1, \theta \rangle, \dots, \langle X_n, \theta \rangle ]^{\top}, \ \theta \in \mathbb{R}^{t \times s},
$
and the response vector  $y = [y_1, \dots, y_n]^\top \in \mathbb{R}^n$. 
We propose the \textit{graph-guided (group) fused lasso (GGFL)} model:

\begin{equation}\label{prob-single}
\min_{\theta\in \mathbb{R}^{t\times s}} \left\{
\begin{aligned}
 f(\theta;{\mathcal{X},y}) := \frac{1}{2} \|y-\mathcal{X}(\theta)\|_2^2 + \lambda_1\|\theta\|_1   + \lambda_t \sum_{i=1}^{t-1} \|\Delta_t^{i}\theta\|_p + \lambda_g \sum_{(j,j')\in {\cal E}} w_{jj'} \|\Delta_g^{jj'}\theta\|_q
\end{aligned}
\right\},
\end{equation}
where 
$ \Delta_t^{i}\theta:=\theta_{i\cdot}-\theta_{(i+1)\cdot}\in\mathbb R^s $,
$ \Delta_g^{jj'}\theta:=\theta_{\cdot j}-\theta_{\cdot j'}\in\mathbb R^t$. {Here, $\mathcal{E}$ and $w_{jj'}$ represent the edge set and corresponding weights of a prespecified spatial graph,} 
$\lambda_1,\lambda_t,\lambda_g>0 $ {are regularization parameters,} 
and $p,q\geq1$ specify the norms.

The regularization terms enforce specific structural assumptions.
First, the $\ell_1$-penalty ($\lambda_1$) induces entry-wise sparsity.
Second, the fused temporal penalty ($\lambda_t$) promotes {similarities}
between adjacent time steps, inducing {a smooth evolution over time}.
Third, the graph-guided spatial penalty ($\lambda_g$) encodes structure via the weighted graph $\mathcal{G}=\{[s],\mathcal{E}, (w_{jj'})\}${, where t}he weights $w_{jj'}$ (e.g., inverse distance) quantify pairwise similarity {to enforce} local {consistency}
among geographically adjacent or related locations.

The choice of norms $p$ and $q$ determines the structure of the penalized differences. Setting $p=q=1$  enforces coordinate-wise sparsity, 
while choosing $p,q > 1$  (e.g., $p=q=2$ in our experiments) induces group-level sparsity, {encouraging the entire difference vector}
to be zero \cite{frecon2016onthefly_mtv}.

\begin{remark}[Relation to the graph total variation (GTV) estimator]
The case of $p=q=1$ in \eqref{prob-single} is conceptually related to the GTV estimator \cite[(4)]{stevens2021graph}{, which} 
minimizes
\begin{equation}\label{eq: gtv}
\min_{\theta\in \mathbb{R}^{t\times s}} \left\{\begin{aligned}
&{\frac{1}{n}}\|y-\mathcal{X}(\theta)\|_2^2 + \lambda_1 \|\theta\|_1  + \lambda_{tv} \sum_{(ij),(i'j')} |{\Sigma}_{(ij),(i'j')}|^{\frac{1}{2}} |\theta_{ij} - {s}_{(ij),(i'j')}\theta_{i'j'}|
\end{aligned}\right\},
\end{equation}
where ${\Sigma}\in\mathbb{S}^{ts}_+$ is the covariance matrix of {the vectorized} predictors $[\operatorname{vec}(X_1),\dots,\operatorname{vec}(X_n)]^\top$ estimated {in a prior step, with}
${s} := \operatorname{sgn}({\Sigma})$. The last penalty term in \eqref{eq: gtv} {represents}
the total variation of the signal $\theta$ {over}
a graph 
defined by $\Sigma$.
While the GTV estimator 
promotes similarity {among} 
predictors via {the weighted differences}
$|\theta_{ij} - {s}_{(ij),(i'j')}\theta_{i'j'}|$,
it treats temporal and spatial {dependencies}
uniformly via a {single} large covariance matrix ${\Sigma}$. In contrast,  our model \eqref{prob-single} separates the spatial and temporal structures via distinct fused and graph-guided penalties, thereby avoiding the estimation of a large covariance matrix.

In fact, \eqref{eq: gtv} aligns with our model \eqref{prob-single} under specific structural assumptions.
Suppose the covariance admits a Kronecker sum form $\Sigma = \Sigma_s\oplus\Sigma_t:= \Sigma_s \otimes I_t + I_s \otimes \Sigma_t$, where $\Sigma_t\in \mathbb S^{t}_+$ and $\Sigma_s\in \mathbb S^{s}_+$ represent temporal and spatial covariance matrices, respectively. 
The GTV penalty (last term in \eqref{eq: gtv}) decouples into additive temporal and spatial components:
\begin{align*}
&\lambda_{tv}
 \sum_{i,i'}|(\Sigma_t)_{ii'}|^{\frac{1}{2}}\|\theta_{i\cdot}-\operatorname{sgn}((\Sigma_t)_{ii'})\theta_{i'\cdot}\|_1+ \lambda_{tv}\sum_{j,j'}|(\Sigma_s)_{jj'}|^{\frac{1}{2}}\|\theta_{\cdot j}-\operatorname{sgn}((\Sigma_s)_{jj'})\theta_{\cdot j'}\|_1.
\end{align*}
Specifically, if we assume
\begin{align*}
    &(\Sigma_t)_{ii'} = 
    \begin{cases}
       { (\lambda_t/2)^2,} & \lvert i - i' \rvert = 1, \\
        0, & \lvert i - i' \rvert >1,
    \end{cases}
    \quad\quad(\Sigma_s)_{jj'} = 
    \begin{cases}
        (\lambda_g w_{jj'})^2, & (j,j') \in \mathcal{E}, \\
        0, & (j,j')\notin\mathcal{E}~\text{and}~j'\ne j,
    \end{cases}
\end{align*}
the GTV formulation \eqref{eq: gtv} is equivalent to our model \eqref{prob-single} with $p=q=1$ {by ignoring the constant scaling of the loss}.
\end{remark}

\subsection{MultiGGFL: Multi-Task Model}\label{sec:multitask}
Given observations $\{(X_k,y_k^{(r)}) \mid k \in [n], r \in [m]\}$ from model~\eqref{eq: model, multi}, we define the response vector of task~$r$ by $y^{(r)} = [y_1^{(r)}, \dots, y_n^{(r)}]^\top \in \mathbb{R}^n$, and collect the regression coefficient $\theta^{(r)}$ into a third-order tensor {$\Theta\in \mathbb{R}^{t\times s\times m}$}
with the $r$-th slice $\Theta^{(r)} = \theta^{(r)}$. Let $\Theta^{(r)}_{ij}$ denote the $(i,j)$-th entry of $\Theta^{(r)}$, and define 
$
\Theta_{[ij]} := [\Theta^{(1)}_{ij}, \Theta^{(2)}_{ij}, \dots, \Theta^{(m)}_{ij}]^\top \in\mathbb{R}^m.
$
See Fig.~\ref{fig:multitask example} for an illustration of these notations. Based on this setup, we propose the \textit{multi-task graph-guided (group) fused lasso (MultiGGFL)} model:
\begin{align}\label{prob-multi}
\min_{\Theta \in \mathbb{R}^{t \times s \times m}} \ \sum_{r=1}^m f(\Theta^{(r)};{\mathcal{X},y^{(r)}}) + \lambda_2 \sum_{i=1}^t\sum_{j=1}^s \|\Theta_{[ij]}\|_2.
\end{align}
Here $f(\Theta^{(r)};{\mathcal{X},y^{(r)}}) $ {denotes}
the single-task GGFL objective \eqref{prob-single} for task $r$, and the second term with parameter $\lambda_2\geq 0$ {promotes}
shared sparsity patterns across tasks by encouraging the grouped coefficients $\Theta_{[ij]}$ to be  zero vectors. This group-wise penalty term is widely used in the learning of multiple tasks, see, e.g., \cite[(3)]{stevens2019graph}, \cite[(3)]{chen2010graph}. When $\lambda_2=0$, this model {decouples into} 
$m$ independent single-task GGFL problems{; furthermore, if}
$m=1$, it reduces to the single-task GGFL~\eqref{prob-single}.

\begin{figure}[htbp]
\centering
\includegraphics[width=0.95\linewidth]{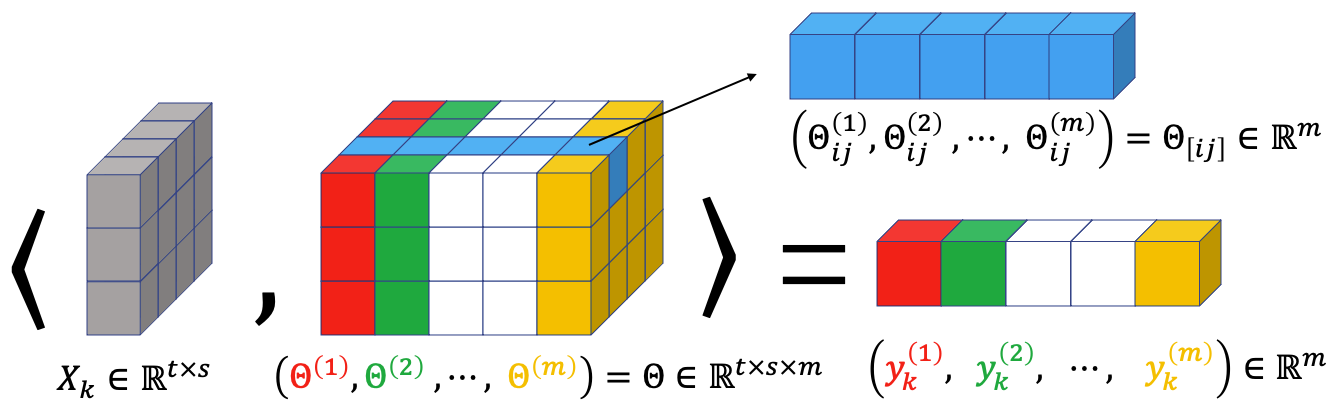}
\caption{Illustration of multi-task model~\eqref{eq: model, multi} and its notations.
Each cube represents a coefficient. 
The highlighted blue cuboid corresponds to the grouped coefficient vector $\Theta_{[ij]}$  across tasks. 
}
\label{fig:multitask example}
\end{figure}

\section{Consistency Analysis}
\label{sec: consistency}

In this section, we establish consistency {analysis}
for the 
multi-task estimator~\eqref{prob-multi}, from which the single-task case follows as a special case.

To this end, we begin by stating a set of regularity assumptions that are standard in the consistency analysis of statistical estimators \cite{bassett1978asymptotic,pollard1991asymptotics,knight2000asymptotics}. Let $\varepsilon_k = [\varepsilon_k^{(1)}, \varepsilon_k^{(2)}, \dots, \varepsilon_k^{(m)}]^\top \in \mathbb{R}^m$ denote the vector of noise terms for the $k$-th observation, and let $\varepsilon^{(r)} = [\varepsilon_1^{(r)}, \varepsilon_2^{(r)}, \dots, \varepsilon_n^{(r)}]^\top \in \mathbb{R}^n$ denote the noise vector for task $r$.

\begin{assumption}\label{assu: iid}
The sequence $\{\varepsilon_k \mid k\in[n]\}$ consists of independent and identically distributed random vectors satisfying $\mathbb{E}[\varepsilon_k] = 0$ and $\mathbb{E}[\varepsilon_k \varepsilon_k^\top] = \Sigma_\varepsilon$.
\end{assumption}

\begin{assumption}\label{assu: regularity}
The empirical covariance operator of the predictors converges to a positive definite limit, i.e.,
\[
\mathcal{M}_n := \frac{1}{n}\mathcal{X}^*\mathcal{X} \;\longrightarrow\; \mathcal{M},
\]
for some self-adjoint positive definite operator $\mathcal{M}: \mathbb{R}^{t\times s} \to \mathbb{R}^{t\times s}$.
\end{assumption}

Under Assumptions~\ref{assu: iid} and~\ref{assu: regularity}, there exists a tensor $\Xi \in \mathbb{R}^{t\times s\times m}$ whose  slices $\Xi^{(r)}$'s are mean-zero Gaussian matrices in $\mathbb{R}^{t\times s}$ satisfying
\begin{equation}\label{eq: cov}
\mathrm{Cov}\big((\Xi^{(r)})_{ij}, (\Xi^{(r')})_{i'j'}\big)
= (\Sigma_\varepsilon)_{rr'}\,
\langle E_{ij}, \mathcal{M}(E_{i'j'}) \rangle,
\end{equation}
for all $ r,r'\in[m],\ i,i'\in[t],\ j,j'\in[s],$ such that, for each $r\in[m]$,
$$
\Xi_n^{(r)}=\frac{1}{\sqrt{n}} \mathcal{X}^* \varepsilon^{(r)} \rightarrow  \Xi^{(r)}
$$ 
in distribution as $n\rightarrow \infty$. Here $E_{ij}\in\mathbb{R}^{t\times s}$ denote the matrix with a one in position $(i,j)$ and zeros elsewhere. 

Moreover, Assumption~\ref{assu: regularity} ensures that, for sufficiently large $n$, the operator $\mathcal{X}^*\mathcal{X}$ is nonsingular, which guarantees the existence of a unique minimizer of the empirical objective in \eqref{prob-multi}. 
{We therefore focus on the following equivalent formulation of the MultiGGFL model~\eqref{prob-multi}:}
\begin{align*}
&\widehat{\Theta}_n = \underset{\Upsilon\in\mathbb R^{t\times s\times m}}{\arg\min}
\ \sum_{r=1}^m \ \frac{1}{2}\|y^{(r)}-\mathcal X(\Upsilon^{(r)})\|_2^2 \\&+\lambda_{t,n}\sum_{r=1}^m\sum_{i=1}^{t-1}\|\Delta_t^{i}\Upsilon^{(r)}\|_{p} +\lambda_{g,n}\sum_{r=1}^m\sum_{(j,j')\in\mathcal E}w_{jj'}\|\Delta_g^{jj'}\Upsilon^{(r)}\|_{q}  \\
&+\lambda_{1,n}\sum_{i=1}^t\sum_{j=1}^s\|\Upsilon_{[ij]}\|_1+\lambda_{2,n}\sum_{i=1}^t\sum_{j=1}^s\|\Upsilon_{[ij]}\|_{2} .
\end{align*}

Our goal is to establish the consistency result of the estimator $\widehat{\Theta}_n$ under the asymptotic regime $n\to\infty$. The following theorem shows that $\widehat{\Theta}_n$ achieves $\sqrt{n}$-consistency.
In particular, when the penalty parameters satisfy $\lambda_{i,n} = o(\sqrt{n})$ for $i \in {1,2,t,g}$, the estimator $\widehat{\Theta}_n$ converges in distribution to the true coefficient tensor ${\Theta}$ in model \eqref{eq: model, multi} as $n \to \infty$.

\begin{theorem}
    Suppose Assumptions \ref{assu: iid} and \ref{assu: regularity} hold. If $\lambda_{i,n}/\sqrt{n}\rightarrow\lambda_{i,0}\geq 0 $ for $i=1,2,t,g.$ Then 
    $$
    \sqrt{n}(\widehat{\Theta}_n- \Theta)\rightarrow \underset{\zeta \in \mathbb{R}^{t\times s\times m}}{\arg\min} F_\infty({\zeta})
    $$ in distribution as $n\rightarrow \infty$, where $\Theta$ is the true coefficient tensor from model \eqref{eq: model, multi}, and
\[
\begin{aligned}
F_\infty(\zeta)
&:= \sum_{r=1}^m \frac{1}{2}\Big( \langle \zeta^{(r)}, \mathcal M(\zeta^{(r)})\rangle - 2\langle \zeta^{(r)},  \Xi^{(r)}\rangle \Big) \\
&\quad + \lambda_{t,0}\sum_{r=1}^m\sum_{i=1}^{t-1} \phi_{p,\Delta_t^{i}{ \Theta}^{(r)}} \big( \Delta_t^{i}\zeta^{(r)}\big)+ \lambda_{g,0}\sum_{r=1}^m\sum_{(j,j')\in\mathcal E} w_{jj'}\,\phi_{q,\Delta_g^{jj'}{ \Theta}^{(r)}} \big( \Delta_g^{jj'}\zeta^{(r)}\big) \\
&\quad + \lambda_{1,0}\,\sum_{i=1}^t\sum_{j=1}^s\phi_{1,{ \Theta}_{[ij]}} \big(\zeta_{[ij]}\big) + \lambda_{2,0}\sum_{i=1}^t\sum_{j=1}^s \phi_{2,{ \Theta}_{[ij]}} \big( \zeta_{[ij]}\big). 
\end{aligned}
\] 
with
\[
    \phi_{\alpha,v}(h) =
    \begin{cases}
    \sum_{k:v_k\neq 0} \operatorname{sgn}(v_k)\,h_k\,
      + \sum_{k:v_k= 0} |h_k|\, & \text{if } \alpha = 1, \\
      \mathbb{I}_{\{v\neq 0\}}v^\top h/\|v\|_2
+ \mathbb{I}_{\{v=0\}}\|h\|_2  & \text{if } \alpha=2.
    \end{cases}
    \] 
    for any vectors $v$ and $h$ of the same length. 
    
    In particular, if $\lambda_{i,n}=o(\sqrt{n})$ for $i=1,2,t,g$, then we have 
    $$
    \sqrt{n}(\widehat{\Theta}_n- \Theta)\rightarrow \operatorname*{argmin}_{\zeta \in \mathbb R^{t\times s \times m}} F_{\infty}(\zeta)=\overline \zeta
    $$ 
    with {$ \bar{\zeta}^{(r) }=\mathcal{M}^{-1} \Xi^{(r)}$} 
    for $r\in [m]$, where $\Xi^{(r)}$'s are mean-zero Gaussian matrices satisfying \eqref{eq: cov}.
\end{theorem}

\begin{proof}
    Denote  $\widehat{\zeta}_n:=\sqrt n\big(\widehat{\Theta}_n-\Theta\big)$ and define a function $F_n:\mathbb{R}^{t\times s\times m}\rightarrow \mathbb{R}$ as
\[
\begin{aligned}
\quad F_n(\zeta)\!=&\! \sum_{r=1}^m\frac{1}{2}\Big( \|y^{(r)}-\mathcal X({ \Theta}^{(r)}+\zeta^{(r)}/\sqrt n)\|_2^2 - \|y^{(r)}-\mathcal X({ \Theta}^{(r)})\|_2^2 \Big) \\
&+ \lambda_{t,n} \sum_{r=1}^m \sum_{i=1}^{t-1} \Big(\|\Delta_t^{i}{ \Theta}^{(r)}+\Delta_t^{i}\zeta^{(r)}/\sqrt n\|_p-\|\Delta_t^{i}{ \Theta}^{(r)}\|_p\Big) \\
&+ \lambda_{g,n} \!\sum_{r=1}^m \!\sum_{(j,j')\in\mathcal E} \!\!\!\!\! w_{jj'}\!\Big(\|\Delta_g^{jj'} \! { \Theta}^{(r)} \!+\!\Delta_g^{jj'}\!\zeta^{(r)}\!/\!\sqrt n\|_q \! -\|\Delta_g^{jj'} \! { \Theta}^{(r)}\|_q\!\Big) \\
&+ \lambda_{1,n} \sum_{i=1}^t\sum_{j=1}^s \Big(\|{ \Theta}_{[ij]}+\zeta_{[ij]}/\sqrt n\|_1-\|{ \Theta}_{[ij]}\|_1\Big) \\
&+ \lambda_{2,n}\sum_{i=1}^t\sum_{j=1}^s \Big(\|{ \Theta}_{[ij]}+\zeta_{[ij]}/\sqrt n\|_{2}-\|{ \Theta}_{[ij]}\|_{2}\Big).
\end{aligned}
\]
By construction, we have 
$$
\underset{\zeta\in \mathbb{R}^{t\times s \times m}}{\arg\min} F_n(\zeta)=\widehat{\zeta}_n=\sqrt n(\widehat\Theta_n-\Theta).
$$

First, by Taylor expansion, the smooth data-fitting term satisfies:
\begin{align*}
&\quad \sum_{r=1}^m\frac{1}{2}\Big( \|y^{(r)}-\mathcal X({ \Theta}^{(r)}+\zeta^{(r)}/\sqrt n)\|_2^2 - \|y^{(r)}-\mathcal X({ \Theta}^{(r)})\|_2^2 \Big)\\
&= \sum_{r=1}^m\frac{1}{2} \Big(-2\langle \zeta^{(r)}, \Xi_n^{(r)}\rangle + \langle \zeta^{(r)}, \mathcal M_n(\zeta^{(r)})\rangle\Big)
\\
&\rightarrow \sum_{r=1}^m\frac{1}{2} \Big( -2\langle \zeta^{(r)}, \Xi^{(r)}\rangle + \langle \zeta^{(r)}, \mathcal M(\zeta^{(r)})\rangle\Big)
\end{align*}
in distribution as $n \to \infty$, using Assumptions \ref{assu: iid} and \ref{assu: regularity}.

Next, for vectors $v$ and $h$ of the same length, we have
\begin{align*}
   \|v+h/\sqrt n\|_1-\|v\|_1=\frac{1}{\sqrt{n}}\sum_{k:v_k\neq 0} \operatorname{sgn}(v_k)\,h_k\,
      + \frac{1}{\sqrt{n}} \sum_{k:v_k= 0} |h_k| =\frac{1}{\sqrt{n}}\,\phi_{1,v}(h), 
\end{align*}
and also
\begin{align*}
\|v+h/\sqrt n\|_2-\|v\|_2&=\frac{1}{\sqrt{n}} \mathbb{I}_{\{v\neq 0\}}v^\top h/\|v\|_2
+ \frac{1}{\sqrt{n}} \mathbb{I}_{\{v=0\}}\|h\|_2 +o\left(\frac{1}{\sqrt{n}}\right) \\
&=\frac{1}{\sqrt{n}}\,\phi_{2,v}(h)+o\left(\frac{1}{\sqrt{n}}\right),
\end{align*}
by  Taylor expansion.
Given the condition that  $\lambda_{i,n}/\sqrt n\to\lambda_{i,0}$ for $i= 1,2,t,g$, we obtain the following component-wise convergence. For  $i\in [t]$, $j\in[s]$,
\[
\begin{aligned}
&\lambda_{1,n}\big(\|{ \Theta}_{[ij]}+\zeta_{[ij]}/\sqrt n\|_1-\|{ \Theta}_{[ij]}\|_1\big)
\ \to\ \lambda_{1,0}\,\phi_{1,{ \Theta}_{[ij]}}(\zeta_{[ij]}),\\
&\lambda_{2,n}\big(\|{ \Theta}_{[ij]}+\zeta_{[ij]}/\sqrt n\|_{2}-\|{ \Theta}_{[ij]}\|_{2}\big)
\ \to\ \lambda_{2,0}\, \phi_{2,{ \Theta}_{[ij]}}(\zeta_{[ij]}),
\end{aligned}
\]
and for   $r\in[m]$, $i\in [t-1]$, $(j,j')\in \mathcal{E}$,
\[
\begin{aligned}\lambda_{t,n}\big(\|\Delta_t^{i}{ \Theta}^{(r)}+\Delta_t^{i}\zeta^{(r)}/\sqrt n\|_p-\|\Delta_t^{i}{ \Theta}^{(r)}\|_p\big)
&\to\ \lambda_{t,0} \, \phi_{p,\Delta_t^{i}{ \Theta}^{(r)}}(\Delta_t^{i}\zeta^{(r)}),\\
\lambda_{g,n} \big(\|\Delta_g^{jj'}{ \Theta}^{(r)}+\Delta_g^{jj'}\zeta^{(r)}/\sqrt n\|_q-\|\Delta_g^{jj'}{ \Theta}^{(r)}\|_q\big)
&\to\ \lambda_{g,0}\, \phi_{q,\Delta_g^{jj'}{ \Theta}^{(r)}}(\Delta_g^{jj'}\zeta^{(r)}),
\end{aligned}
\]
as $n\to \infty $. Therefore, $F_n(\zeta)\rightarrow F_\infty(\zeta)$ in distribution for every fixed $\zeta\in \mathbb{R}^{t\times s\times m}$.

Since $F_\infty$ is strictly convex with a unique minimizer and $F_n$ is convex, the convexity lemma \cite{geyer1996asymptotics} implies
$$
\sqrt n(\widehat\Theta_n-\Theta)=\underset{\zeta\in \mathbb{R}^{t\times s \times m}}{\arg\min} F_n(\zeta) \rightarrow \underset{\zeta\in \mathbb{R}^{t\times s \times m}}{\arg\min} F_\infty(\zeta)
$$ 
in distribution as $n\rightarrow \infty$. 
In particular, when  $\lambda_{i,0}=0$ for $i=1,2,t,g$, the unique minimizer of $F_\infty$ takes the explicit form:
$$
\overline \zeta:=\operatorname*{argmin}_{\zeta \in \mathbb R^{t\times s \times m}} F_{\infty}(\zeta) $$
with $ \overline \zeta^{(r) }=\mathcal{M}^{-1} \Xi^{(r)}$ for $ r\in [m].$ This completes the proof.
\end{proof}

\section{A Halpern Peaceman-Rachford Method}
\label{sec: alg}
In this section, we employ the Halpern Peaceman–Rachford (HPR) method to solve the single-task learning problem \eqref{prob-single} and its multi-task counterpart \eqref{prob-multi}. Since 
the single-task model 
\eqref{prob-single} arises as the special case of the multi-task model~\eqref{prob-multi} with $m=1$ and $\lambda_{2}=0$, we focus on the latter problem. 

For  
subsequent algorithmic development, it is convenient to rewrite \eqref{prob-multi} into the following compact form:
\begin{equation}\label{eq: prob-multi-refor-uncons}
\min 
\ \ell(\Theta)
+ \Omega(\mathcal P\Theta)
+ \Phi(\mathcal Q\Theta)
+ \Psi(\Theta),
\end{equation}
where the operators $\mathcal P$, $\mathcal Q$, and the functions $\ell,\Omega,\Phi,\Psi$ are introduced below.
First, the squared loss function is given by
\[
\ell(\Theta)
=\frac{1}{2}\sum_{r=1}^{m}\|y^{(r)}-\mathcal{X}(\Theta^{(r)})\|_{2}^{2},
\]
for $\Theta = [\Theta^{(1)},\dots,\Theta^{(m)}]\in \mathbb{R}^{t\times s\times m}$.
Second, {temporal smoothness is imposed by the $(t-1)\times t$  difference matrix $P\in\mathbb{R}^{(t-1)\times t}$, based on which  we  define the linear operator $\mathcal P:\mathbb R^{t\times s\times m}\to\mathbb R^{(t-1)\times s\times m}$:
\[
P=\begin{bmatrix}
1 & -1 &  &  & \\
 & 1 & -1 &  & \\
 & & \ddots & \ddots & \\
 &  &  & 1 & -1
\end{bmatrix}, \; \mathcal P\Theta= [P\Theta^{(1)},\dots, P\Theta^{(m)}].
\]
The temporal regularizer  $\Omega:\mathbb{R}^{(t-1)\times s\times m} \to \mathbb{R}$ then takes the form
\begin{align*}
\Omega(W) = \lambda_{t} \sum_{r=1}^m\sum_{i=1}^{t-1}\| W^{(r)}_{i\cdot} \|_{p}, 
\quad   W = [W^{(1)},\dots,W^{(m)}].
\end{align*}
Third, spatial smoothness is imposed by the node-arc incidence matrix of the spatial graph 
$\mathcal G=([s],\mathcal E,(w_{jj'}))$.
Let $|\mathcal E|$ be the number of edges, indexed lexicographically via $\iota:\mathcal E\to[|\mathcal E|]$.
The (unweighted) incidence matrix $B\in\mathbb{R}^{s\times|\mathcal E|}$ and the linear operator $\mathcal Q:\mathbb R^{t\times s\times m}\to\mathbb R^{t\times |\mathcal E|\times m}$ are defined by
\[
B_{k,\,\iota(j,j')}=
\begin{cases}
1, & k=j,\\
-1, & k=j',\\
0, & \text{otherwise},
\end{cases}\;\;
\mathcal Q\Theta=[\Theta^{(1)} B,\dots,\Theta^{(m)} B].
\]
The spatial regularizer $\Phi:\mathbb{R}^{t\times |\mathcal E|\times m} \to \mathbb{R}$ is then
\begin{align*}
\Phi(Z)=\lambda_{g} \sum_{r=1}^m \sum_{j,j'} w_{jj'} \|Z^{(r)}_{\cdot\, \iota(j,j')}\|_q,  
\quad  Z = [Z^{(1)},\dots,Z^{(m)}].
\end{align*}
Lastly, the cross-task sparse group Lasso regularizer is given by
\[
\Psi(\Theta)
=\sum_{i=1}^{t}\sum_{j=1}^{s}
\Bigl(\lambda_{1}\,\|\Theta_{[ij]}\|_{1}
+\lambda_{2}\,\|\Theta_{[ij]}\|_{2}\Bigr),
\]
where the $\ell_1$ term  promotes sparsity at the element-wise level, {and}  
the $\ell_2$ term 
{encourages} a shared sparsity support across tasks.
}

\subsection{An HPR Method for Solving MultiGGFL}
To apply the HPR method, we rewrite problem~\eqref{eq: prob-multi-refor-uncons} in an equivalent constrained form by introducing three sets of slack variables that separate the loss, temporal regularizer (slack variable $W\in\mathbb R^{(t-1)\times s\times m}$), spatial regularizer (slack variable $Z\in\mathbb R^{t\times|\mathcal E|\times m}$), and cross-task sparse group regularizer (slack variable $U\in\mathbb R^{t\times s\times m}$). Specifically, we consider
\begin{align}
\min
&\quad  \ell(\Theta)
+ \Omega(W)
+ \Phi(Z)
+ \Psi(U)
\label{prob-multi-refor} \\[1mm]
\text{s.t.}
&\quad \mathcal P\Theta - W = 0, \
  \mathcal Q\Theta - Z = 0, \
  \Theta - U = 0. \notag
\end{align}
Let  
\(S\in\mathbb R^{(t-1)\times s\times m}\),  
\(T\in\mathbb R^{t\times|\mathcal E|\times m}\),  
and \(R\in\mathbb R^{t\times s\times m}\)  
be the Lagrange multipliers associated with the three linear constraints.  
The KKT system associated with \eqref{prob-multi-refor} for a primal-dual pair 
\( H:=(\Theta, W, Z, U, S, T, R) \)  is
\begin{equation}\label{eq:KKT}
\begin{aligned}
&0=\nabla\ell(\Theta)+\mathcal P^{*} S+\mathcal Q^{*} T + R,\\[1mm]
& 0\in \partial\Omega(W)-S,
\quad  0\in \partial\Phi(Z)-T,
\quad 0\in \partial\Psi(U)-R,\\[1mm]
&\mathcal P\Theta - W = 0,
\quad \mathcal Q\Theta - Z = 0,
\quad \Theta - U = 0.
\end{aligned}
\end{equation}
Since the primal problem \eqref{prob-multi-refor} has a solution, it follows from \cite[Corollary 28.3.1]{rockafellar1997convex} that the KKT system admits a solution. 
For $\sigma>0$, the augmented Lagrangian function associated with~\eqref{prob-multi-refor} is
\begin{align*}
& \mathcal L_{\sigma}(\Theta,W,Z,U;S,T,R)
=
   \ell(\Theta)
 + \Omega(W)
 + \Phi(Z)
 + \Psi(U)   \\[1mm]
& + \langle S, \mathcal P\Theta - W\rangle
 + \langle T, \mathcal Q\Theta - Z\rangle
 + \langle R, \Theta - U\rangle  \\[1mm]
 & + \frac{\sigma}{2}\Bigl(
     \|\mathcal P\Theta - W\|_F^{2}
   + \|\mathcal Q\Theta - Z\|_F^{2}
   + \|\Theta - U\|_F^{2}
   \Bigr).
\end{align*}

Based on this setup, we now describe the HPR method for solving~\eqref{prob-multi-refor}, summarized in Algorithm~\ref{alg:hpr}. Its global convergence follows from \cite[Corollary 3.5]{sun2025accelerating}. For completeness, we state it in Theorem~\ref{thm:thm-convergence-1}.
\begin{algorithm}[htbp]
  \caption{An HPR method for solving~\eqref{prob-multi-refor}: {$(\overline{H}_{k+1},k+1) = \textbf{HPR}(H_0, \sigma)$.} 
  }
  \label{alg:hpr}
  \begin{algorithmic}
    \STATE \textbf{Input:} 
    \(
      H_0=(\Theta_0,W_0,Z_0,U_0,S_0,T_0,R_0)\),
    \(  \sigma>0.
    \)

    \FOR{$k=0,1,\ldots$}

    \STATE \textbf{Step 1.}  
      $\displaystyle  
      \overline\Theta_{k+1}
      =
      \underset{\Theta}{\arg\min}\,
      \mathcal L_{\sigma}(
     \Theta,W_k,Z_k,U_k,S_k,T_k,R_k).
      $
 \vspace{0.5mm}
 
    \STATE \textbf{Step 2.}  Update multipliers: 
    \begin{align*}
    &\overline S_{k+1} = S_k + \sigma (\mathcal P\overline\Theta_{k+1}-W_k),\\
     & \overline T_{k+1} = T_k + \sigma (\mathcal Q\overline\Theta_{k+1}-Z_k),\\
     & \overline R_{k+1} = R_k + \sigma (\overline\Theta_{k+1}-U_k).
    \end{align*}

    \STATE \textbf{Step 3.} Update slack variables:
    \begin{align*}
      &\quad (\overline W_{k+1},\overline Z_{k+1},\overline U_{k+1})
      =
      \underset{W,Z,U}{\arg\min}\
      \mathcal L_{\sigma}\bigl(
        \overline\Theta_{k+1},W,Z,U;
        \overline S_{k+1},\overline T_{k+1},\overline R_{k+1}
      \bigr).
   \end{align*}

    \STATE \textbf{Step 4.} {Extrapolation:}
      $\displaystyle
      \widehat H_{k+1} = 2\,\overline H_{k+1} - H_k.
      $
  \vspace{2mm}
    \STATE \textbf{Step 5.} {Halpern iteration:}
      $$\displaystyle
      H_{k+1} = \frac{1}{k+2} H_0 + \frac{k+1}{k+2} \widehat H_{k+1}.
      $$

    \ENDFOR

    \STATE \textbf{Output:} 
    \(
    {(\overline H_{k+1}, k+1).}
    \)

  \end{algorithmic}
\end{algorithm}

\begin{theorem} \label{thm:thm-convergence-1}
The sequence 
$
\{\overline{H}_k\} =\{
(\overline{\Theta}_k, \overline{W}_k, \overline{Z}_k, \overline{U}_k, \overline{S}_k, \overline{T}_k, \overline{R}_k)\}
$
generated by Algorithm~\ref{alg:hpr} converges to 
$
H^* = (\Theta^*, W^*, Z^*, U^*, S^*, T^*, R^*)
$,
{which satisfies the KKT system \eqref{eq:KKT}.}
Namely, $(\Theta^*, W^*, Z^*, U^*)$ is an optimal solution to problem~\eqref{prob-multi-refor}, and $(S^*, T^*, R^*)$ is an optimal solution to its dual problem.
\end{theorem}
\begin{proof}
We first rewrite problem~\eqref{prob-multi-refor} in a standard two-block form \cite[(1.1)]{sun2025accelerating}:
\[
\min_{y,z} \quad f_1(y) + f_2(z)
\quad \text{s.t.} \quad
B_1 y + B_2 z = 0,
\]
where 
$
f_1(y) := \Omega( W) + \Phi( Z) + \Psi( U)$, $
f_2(z) := \ell(\Theta)
$, with
\[
y:=\begin{bmatrix}
     W\\  Z\\  U
\end{bmatrix},\quad
z := \Theta,\quad
B_1 = -I,
\quad
B_2 =
\begin{bmatrix}
\mathcal{P}\\
\mathcal{Q}\\
{I}
\end{bmatrix}.
\]
Since both $B_1^* B_1$ and $B_2^* B_2$ are positive definite, Assumptions~1 and~2 in \cite{sun2025accelerating} hold.  The global convergence then follows  directly from ~\cite[Corollary~3.5]{sun2025accelerating}.
\end{proof}

\begin{remark}
Although Algorithm~\ref{alg:hpr} follows the update steps of the accelerated preconditioned ADMM (pADMM) framework~\cite[Algorithm 3.1]{sun2025accelerating} with parameters $\alpha=\rho=2$ and $\mathcal{T}_1=\mathcal{T}_2=0$,
it has been shown in~\cite[Proposition 4]{chen2024hpr} that pADMM with this parameter setting is mathematically equivalent to the HPR method in~\cite{zhang2022efficient}.
For this reason, we refer to Algorithm~\ref{alg:hpr} as the HPR method throughout this paper.
\end{remark}

We measure the convergence of Algorithm~\ref{alg:hpr} via the KKT residual and the objective value gap.
The residual mapping $\mathcal{R}$ associated with the KKT system~\eqref{eq:KKT} is defined as:
\begin{equation*}
\mathcal R(H)
:=
\left(
\begin{array}{c}
\Theta - \operatorname{prox}_{\ell} (\Theta - \mathcal P^* S - \mathcal Q^* T -R)\\[2pt]
W - \operatorname{prox}_{\Omega}(W + S)\\[2pt]
Z - \operatorname{prox}_{\Phi}(Z + T)\\[2pt]
U - \operatorname{prox}_{\Psi}(U + R)\\[2pt]
\mathcal P\Theta - W\\[2pt]
\mathcal Q\Theta - Z\\[2pt]
\Theta - U
\end{array} \right),
\end{equation*}
for $H=(\Theta, W, Z, U, S, T, R)$.
A point \(H^*\) satisfies \eqref{eq:KKT} if and only if \(\mathcal R(H^*) = 0\).
We also measure the optimality gap for the objective function $F(\Theta,W,Z,U)
:= \ell(\Theta)
+ \Omega(W)
+ \Phi(Z)
+ \Psi(U)$ in \eqref{prob-multi-refor}:
\begin{align*}
&\quad \Delta F(\overline\Theta_k,\overline W_k,\overline Z_k,\overline U_k)
:= F(\overline\Theta_k,\overline W_k,\overline Z_k,\overline U_k)
   - F(\Theta^*, W^*, Z^*, U^*),
\end{align*}
where \((\Theta^*,W^*,Z^*,U^*)\) denotes the limit point of the sequence 
\(\{(\overline\Theta_k,\overline W_k,\overline Z_k,\overline U_k)\}\) established in Theorem~\ref{thm:thm-convergence-1}, and thus is an optimal solution of~\eqref{prob-multi-refor}.

Based on \cite[Theorem~3.7]{sun2025accelerating}, Algorithm~\ref{alg:hpr} achieves a non-ergodic $\mathcal O(1/k)$ convergence rate, as stated below.

\begin{theorem}
Let \(\{\overline{ H}_k\}=
\bigl\{(\overline{\Theta}_k,\overline{ W}_k,\overline{ Z}_k,\overline{ U}_k,\overline{ S}_k,\overline{ T}_k,\overline{ R}_k)\bigr\}
\) be the sequence generated by Algorithm~\ref{alg:hpr}, and let
\(
{ H}^*=(\Theta^*, W^*, Z^*, U^*, S^*, T^*, R^*)
\)
be its limit point. Then for all $k\ge 0$, the following bounds hold:
$$
 \bigl\|\mathcal{R}(\overline{{ H}}_k)\bigr\|_F
\leq { c \left(1+\frac{1}{\sigma}\right)\frac{1}{k+1},}
$$
and
$$ -\frac{c d}{\sigma} \frac{1}{k+1} \leq \Delta F (\overline{\Theta}_k,\overline{ W}_k,\overline{ Z}_k,\overline{U}_k) \leq
\frac{(3c + d) c}{\sigma} \frac{1}{k+1}, 
$$
where 
$$
{c:=}\left\|
\sigma
\begin{pmatrix}
    W_0 - W^*\\Z_0-Z^*\\U_0-U^*
\end{pmatrix}  - 
\begin{pmatrix}
    S_0-S^*\\T_0-T^*\\R_0-R^*
\end{pmatrix} 
\right\|_{F}
$$
measures the distance {from} the initial point to the limit point, and  $d = \|(S^*,T^*,R^*)\|_F$ is the norm of the optimal multipliers.
\end{theorem}

\subsection{Per-iteration Updates and Costs}
\label{sec: updata_complexity}

This subsection details the explicit update formulas and computational costs for each step of Algorithm~\ref{alg:hpr}. 

Step 1 of Algorithm \ref{alg:hpr} can be solved by setting the gradient of $ \mathcal L_\sigma$ with respect to each $\Theta^{(r)}$ to be zero. A direct calculation shows that each ${\overline\Theta}_{k+1}^{(r)}$ can be computed by solving the linear system
\begin{align}
\mathcal{X}^{*} \mathcal{X}({\overline\Theta_{k+1}^{(r)}}) 
&+ \sigma (I_{t} + P^{\top}P){\overline\Theta_{k+1}^{(r)}}
+ \sigma {\overline\Theta_{k+1}^{(r)}}(Q Q^{\top})
= b_k^{(r)}, \label{eq: linear system}
\end{align}
where 
\begin{align*}
b_k^{(r)} &:=
\mathcal{X}^{*} y^{(r)}
+ P^{\top}\big(\sigma  W_k^{(r)} -   S_k^{(r)}\big) + \big(\sigma  Z_k^{(r)} -   T_k^{(r)}\big)Q^{\top}
+ \sigma   U_k^{(r)} -   R_k^{(r)},
\end{align*}
and $\mathcal{X}^{*}$ denotes the adjoint of $\mathcal{X}$. {We denote the temporal and spatial Laplacians as $L_P = P^\top P$ and $L_Q = Q Q^\top$, respectively, with their entries defined as:}
\begin{align*}
&(L_P)_{ii'} =
\begin{cases}
2, & 2\le i=i'\le t-1,\\
1, & i=i'=1~\text{or}~i=i'=t,\\
-1, & |i-i'|=1,\\
0, & \text{otherwise},
\end{cases} \qquad(L_Q)_{jj'} =
\begin{cases}
\deg(v_j), & j'=j,\\
-1, & j'\ne j~\text{and}~(j,j')\in\mathcal{E},\\
0, & \text{otherwise},
\end{cases}
\end{align*}
{where $\deg(v_j)$ denotes the degree of vertex $v_j$ in the spatial graph $\mathcal{G}$.}
Vectorizing \eqref{eq: linear system} yields the linear system
\begin{align}
&\big(X^\top X + \sigma I_s \otimes (I_t + L_P) 
      + \sigma L_Q \otimes I_t\big)\,\mathrm{vec}({\overline\Theta_{k+1}^{(r)}}) = \mathrm{vec}(b_k^{(r)}),\label{eq: vec linear system}
\end{align}
where 
$X=[\mathrm{vec}(X_1),\mathrm{vec}(X_2),\ldots,\mathrm{vec}(X_n)]^\top
   \in \mathbb R^{n\times ts}$.
It can be solved efficiently by direct solvers 
(e.g., Cholesky factorization for small or medium scales) or  iterative solvers such as the conjugate gradient method.

In {Step~3}, the variables $\mathcal W$, $\mathcal Z$, and $\mathcal U$ are updated via the following proximal mappings:
\[
\begin{aligned}
\overline W_{k+1}
&= \operatorname{prox}_{\Omega/\sigma } \left(\mathcal P\,\overline\Theta_{k+1}
      + \sigma^{-1}\,\overline S_{k+1}\right),\\[1mm]
\overline Z_{k+1}
&= \operatorname{prox}_{\Phi/\sigma } \left(\mathcal Q\,\overline\Theta_{k+1}
      + \sigma^{-1}\,\overline T_{k+1}\right),\\[1mm]
\overline U_{k+1}
&= \operatorname{prox}_{\Psi/\sigma } \left(\overline\Theta_{k+1}
      + \sigma^{-1}\,\overline R_{k+1}\right),
\end{aligned}
\]
which admit closed-form solutions.  
For $\nu>0$, $U\in\mathbb R^{t\times s\times m}$, $W\in\mathbb R^{(t-1)\times s\times m}$, and 
$Z\in\mathbb R^{t\times|\mathcal E|\times m}$, they are given  by
\[
\begin{aligned}
& \bigl[\operatorname{prox}_{\Omega/\sigma}(W)\bigr]^{(r)}_{i\cdot}
   = \operatorname{prox}_{\lambda_t/\sigma\|\cdot\|_p}\Bigl(W^{(r)}_{i\cdot}\Bigr),\\[3pt]
& \bigl[\operatorname{prox}_{\Phi/\sigma}(Z)\bigr]^{(r)}_{\cdot\,\iota(j,j')}
   = \operatorname{prox}_{\lambda_g w_{jj'}/\sigma\|\cdot\|_q}\Bigl(Z^{(r)}_{\cdot\,\iota(j,j')}\Bigr), \\[3pt]
&   \bigl[\operatorname{prox}_{\Psi/\sigma}(U)\bigr]_{[ij]} 
   = \operatorname{prox}_{\lambda_2/\sigma\|\cdot\|_2}
       \Bigl(\operatorname{prox}_{\sigma^{-1}\lambda_1\|\cdot\|_1}(U_{[ij]})\Bigr).
\end{aligned}
\]

With these update rules, we are ready {to} quantify the computational cost of one loop of Algorithm~\ref{alg:hpr}. 
First, a one-time preprocessing step is required to form the coefficient matrix $X^\top X + \sigma I_s\otimes(I_t + L_P ) + \sigma L_Q\otimes I_t \in \mathbb{S}^{ts}$ and the vectors $\{ \mathcal{X}^* y^{(r)} \mid r\in[m]\}$ appearing on the right-hand side of the linear system~\eqref{eq: vec linear system}. These computations require $\mathcal{O}(t^2s^2n)$ and $tsmn$ operations, respectively. Subsequently, factorizing the resulting $ts \times ts$ matrix using the Cholesky decomposition incurs additional $\mathcal{O}(t^3s^3)$ operations. 
Next, within each iteration of Algorithm~\ref{alg:hpr},  the right-hand side $b_k^{(r)}$ of the linear system~\eqref{eq: vec linear system}  can be formed in $\mathcal{O}(ts + t|\mathcal E|)$ operations for each $r\in [m]$, and solving the system using the precomputed Cholesky factors then costs $\mathcal{O}(t^2s^2)$ operations per right-hand side. Thus, Step~1 of Algorithm~\ref{alg:hpr} requires $\mathcal{O}(tsm + t|\mathcal E|m+t^2s^2m)$ operations per loop. The remaining proximal updates for $W$, $Z$, and $U$ are lightweight, each costing at most $\mathcal{O}(tsm + t|\mathcal E|m)$, which is dominated by the linear-system solve. In summary, the dominant computational cost in each loop of Algorithm~\ref{alg:hpr} comes from solving the linear system in Step~1, which requires $\mathcal{O}(tsm + t|\mathcal E|m+t^2s^2m)$ operations, while all remaining updates contribute only lower-order costs.

\subsection{Practical Acceleration}
\label{sec: acceleration}
To enhance practical efficiency, we incorporate two strategies  from \cite{chen2024hpr,chen2025hpr}:
a restart mechanism that adaptively reinitializes the sequence to accelerate convergence, and an adaptive $\sigma$ update that dynamically balances the primal and dual residuals.

First, we monitor the progress of the $(k+1)$-th iteration via the residual quantity:
\[
c_{k}
:=
\left\|\sigma\begin{pmatrix}
\widehat{{W}}_{k+1}-{W}_{k}\\
\widehat{{Z}}_{k+1}-{Z}_{k}\\
\widehat{{U}}_{k+1}-{U}_{k}
\end{pmatrix}-\begin{pmatrix}
\widehat{{S}}_{k+1}-{S}_{k}\\
\widehat{{T}}_{k+1}-{T}_{k}\\
\widehat{{R}}_{k+1}-{R}_{k}
\end{pmatrix} \right\|_{F}.
\]
We check $c_k$ every 50 iterations, and the HPR loop terminates (restarts) if any of the following hold:
\begin{equation}\label{eq:restart}
\begin{aligned}
&c_{k-1}<c_{k}\le \alpha_{1}c_{0}, && \alpha_{1}\in(0,1),\\
&c_{k}\le \alpha_{2}c_{0}, && \alpha_{2}\in(0,\alpha_{1}),\\
&k\ge \alpha_{3}\,t, && \alpha_{3}>0.
\end{aligned}
\end{equation}
In our experiments, we set $\alpha_1 = 0.6$, $\alpha_2 = 0.2$, and $\alpha_3 = 0.25$.
The three conditions trigger a restart when: $c_k$ stops decaying (local stagnation);  $c_k$ drops sufficiently (successful decay); or the inner iteration  count $k$ becomes excessive relative to the total history $t$.

Upon restart, we update $\sigma$ to balance primal and dual progress. We measure the relative motion of the primal and dual variables by
\begin{align}
\Delta_{p}:=\left\|
\begin{pmatrix}
\overline{ W}_{\kappa+1}-\overline{ W}_{\kappa}\\
\overline{ Z}_{\kappa+1}-\overline{ Z}_{\kappa}\\
\overline{ U}_{\kappa+1}-\overline{ U}_{\kappa}
\end{pmatrix}\right\|_{F},\quad
\Delta_{d}:=\left\|
\begin{pmatrix}
\overline{ S}_{\kappa+1}-\overline{ S}_{\kappa}\\
\overline{ T}_{\kappa+1}-\overline{ T}_{\kappa}\\
\overline{ R}_{\kappa+1}-\overline{ R}_{\kappa}
\end{pmatrix}\right\|_{F},
\label{eq:delta-def}
\end{align}
and set $\sigma_{\kappa +1}= {\Delta_{d}}/{\Delta_{p}}$.
This practical implementation is summarized in Algorithm \ref{alg: hpr-restart}.

\begin{algorithm}[htbp]
  \caption{Practical HPR for solving~\eqref{prob-multi-refor}}
  \label{alg: hpr-restart}
  \begin{algorithmic}[1]
    \STATE \textbf{Input:}  $\overline{H}_{0}$, $\sigma_{0}=1$. Set total iteration count {$t\gets 0$}.
    \FOR{outer iteration $\kappa=0,1,\dots$}
      \STATE \textbf{Step 1.} Run $(\overline{H}_{\kappa+1}, {\Delta t}) \leftarrow \mathbf{HPR}(\overline{H}_{\kappa}, \sigma_{\kappa})$, checking conditions {in} \eqref{eq:restart} every 50 iterations. Terminate when met.
      \STATE \textbf{Step 2.} 
      Set $\sigma_{\kappa+1} \leftarrow {\Delta_{d}}/{\Delta_{p}}$, and {$t\gets t+ \Delta t$}{, where $\Delta_{d}, \Delta_{p}$ are computed via \eqref{eq:delta-def}.}
    \ENDFOR
    \STATE \textbf{Output:} {$\overline{H}_{\kappa+1}$.}
  \end{algorithmic}
\end{algorithm}

\section{Numerical Experiments}
\label{sec: experiment}
In this section, we conduct extensive numerical experiments designed to examine both the statistical and computational aspects of our proposed estimators, compared with the state-of-the-art GTV and its variant {\cite{stevens2019graph,li2020graph,stevens2021graph}.} 
We first evaluate their statistical behavior through simulation studies. 
Next, we investigate the computational efficiency and scalability of the proposed HPR algorithm in both single- and multi-task settings. Finally, we demonstrate the practical utility of our estimators in a real application on winter precipitation forecasting in the Southwestern United States (SWUS).

Experiments were conducted in MATLAB R2024a on an Intel i7-860 (2.80 GHz) CPU with 16 GB RAM. Unless otherwise noted, Algorithm~\ref{alg: hpr-restart} terminates if the total iteration count reaches $2000$ or the normalized KKT residual satisfies:
\begin{equation*}
\eta_{\mathrm{KKT}}:=\max\{R_d,R_p\}\le 10^{-4},
\end{equation*}
where the primal ($R_p$) and dual ($R_d$)  residuals are defined as:
\begin{align*}
R_{p} = \max\Bigg\{&
\frac{\|\mathcal P\overline{\Theta}_{\kappa}-\overline{W}_{\kappa}\|_{F}}{1+\|\overline{W}_{\kappa}\|_{F}},
\frac{\|\mathcal Q\overline{\Theta}_{\kappa}-\overline{Z}_{\kappa}\|_{F}}{1+\|\overline{Z}_{\kappa}\|_{F}},
\frac{\|\overline{\Theta}_{\kappa} - \overline{U}_{\kappa}\|_{F}}{1+\|\overline{U}_{\kappa}\|_{F}}
\Bigg\}, \\
R_{d} = \max\Bigg\{
&\frac{\bigl\|\nabla \ell(\overline{\Theta}_{\kappa}) +  \mathcal P^{*}\overline{S}_{\kappa}  + \mathcal Q^{*}\overline{T}_{\kappa}+ \overline{R}_{\kappa} \bigr\|_{F}}{1+\|\overline{R}_{\kappa}\|_{F}},\frac{\bigl\|\overline{W}_{\kappa} - \operatorname{prox}_{\Omega}(\overline{W}_{\kappa}  + \overline{S}_{\kappa})\bigr\|_{F}}{1+\|\overline{W}_{\kappa}\|_{F}},\\[3pt]
&\frac{\bigl\|\overline{Z}_{\kappa} - \operatorname{prox}_{\Phi}(\overline{Z}_{\kappa}  + \overline{T}_{\kappa})\bigr\|_{F}}{1+\|\overline{Z}_{\kappa}\|_{F}},\frac{\bigl\|\overline{U}_{\kappa} - \operatorname{prox}_{\Psi}(\overline{U}_{\kappa} + \overline{R}_{\kappa})\bigr\|_{F}}{1+\|\overline{U}_{\kappa}\|_{F}}\Bigg\}.
\end{align*}

In our experiments, we consider a setting where the $s$ spatial locations lie on an $s_1\times s_2$ grid ($s=s_1s_2$). For simplicity, we set the weights in our models \eqref{prob-single} and \eqref{prob-multi} to reflect 4-neighbor connectivity: 
\begin{equation} \label{def:w4}
w_{jj'} = \begin{cases} 1, & \text{if } j \text{ and } j' \text{ are adjacent,} \\ 0, & \text{otherwise,} 
\end{cases} 
\end{equation}
{which connect}
immediate horizontal and vertical neighbors in the spatial space.

\subsection{Synthetic Data}
We first compare the performance of the GGFL estimator with that of GTV \cite{stevens2021graph}, and then assess the computational scalability of the HPR algorithm {(Algorithm~\ref{alg: hpr-restart})} as the dimensions of the problem increase.

\begin{figure*}[htbp]
\centering
\includegraphics[width=0.28\linewidth]{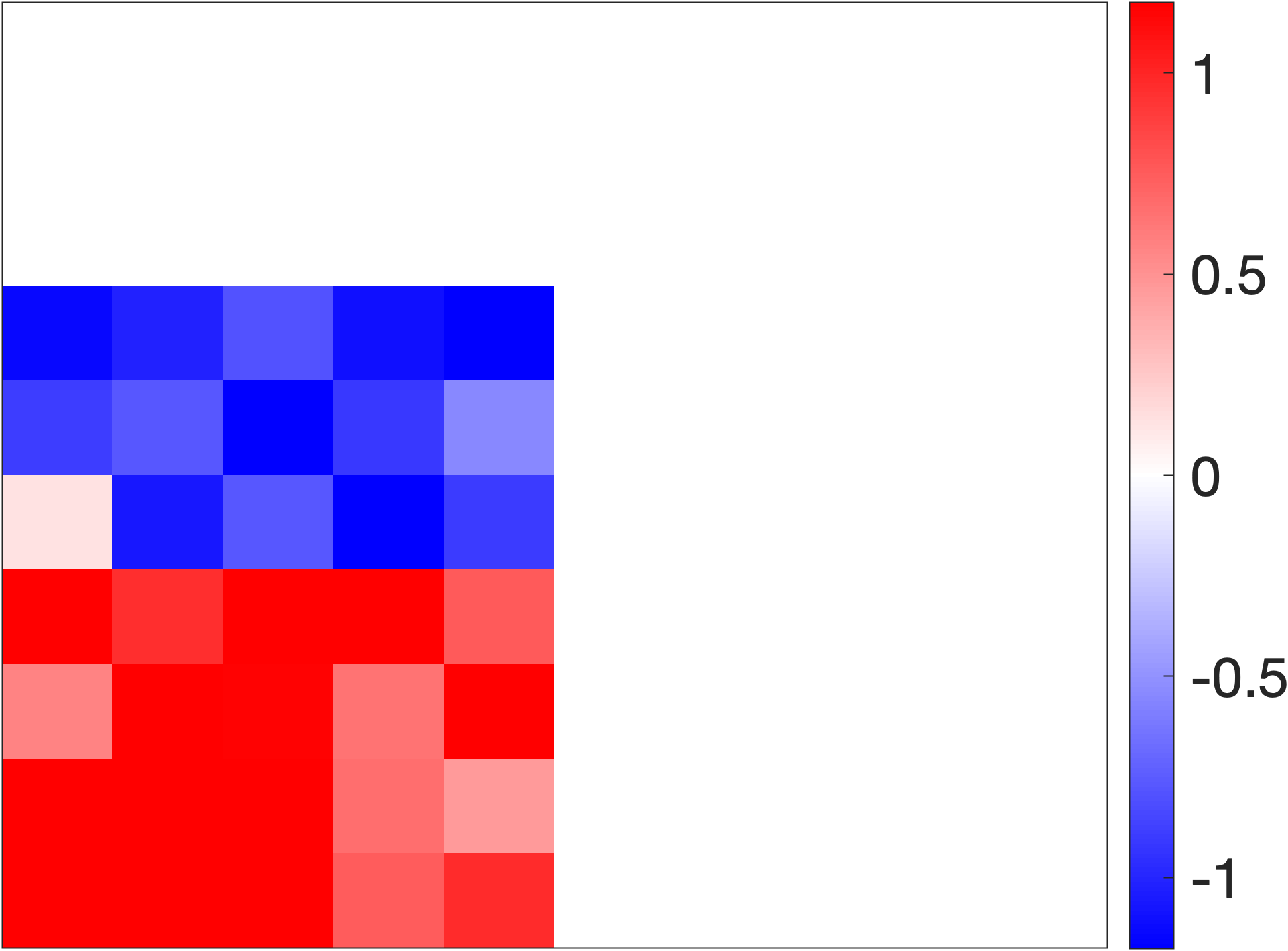}
\hspace{0.8cm}
\includegraphics[width=0.28\linewidth]{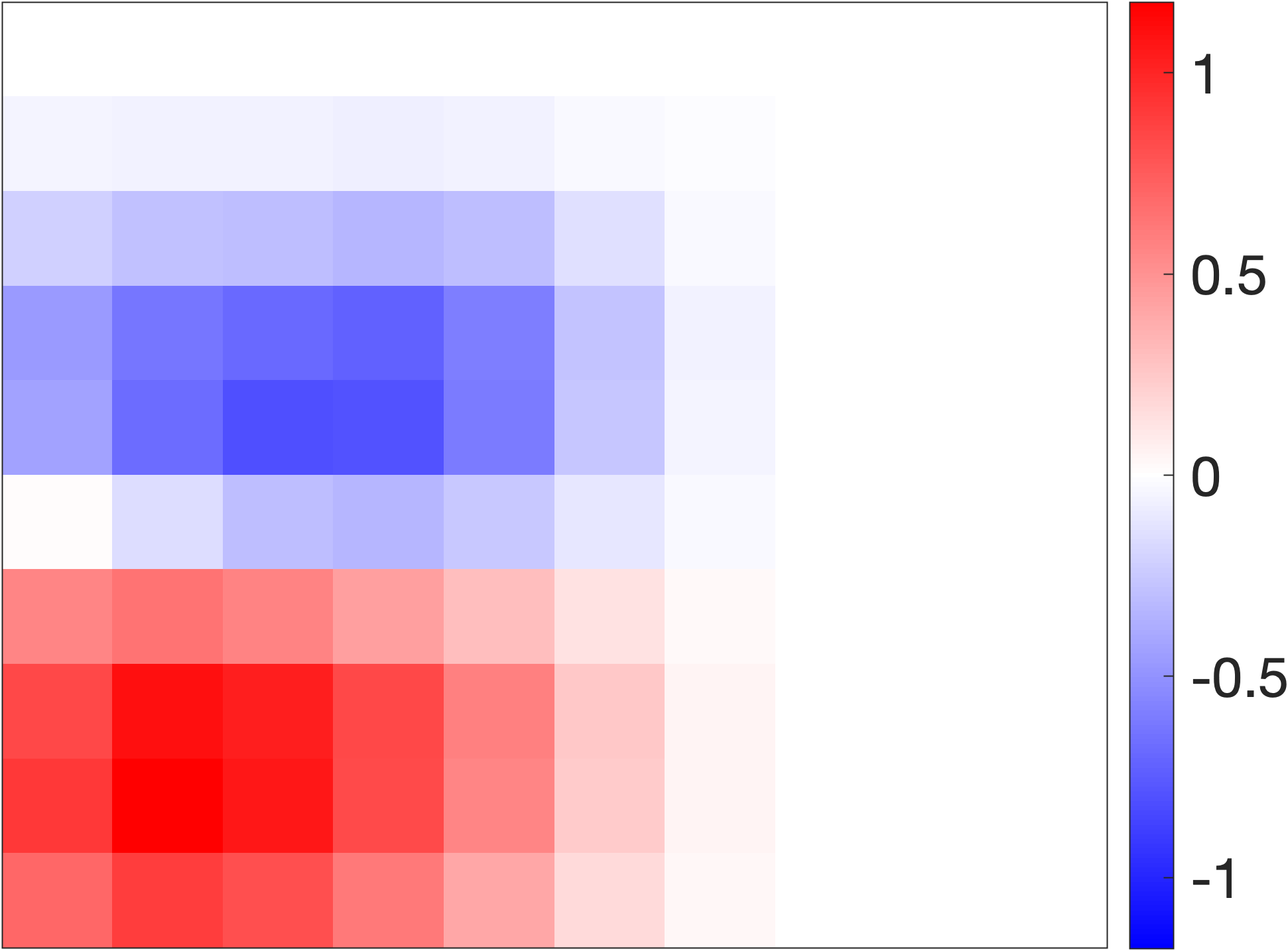}
\hspace{0.8cm}
\includegraphics[width=0.28\linewidth]{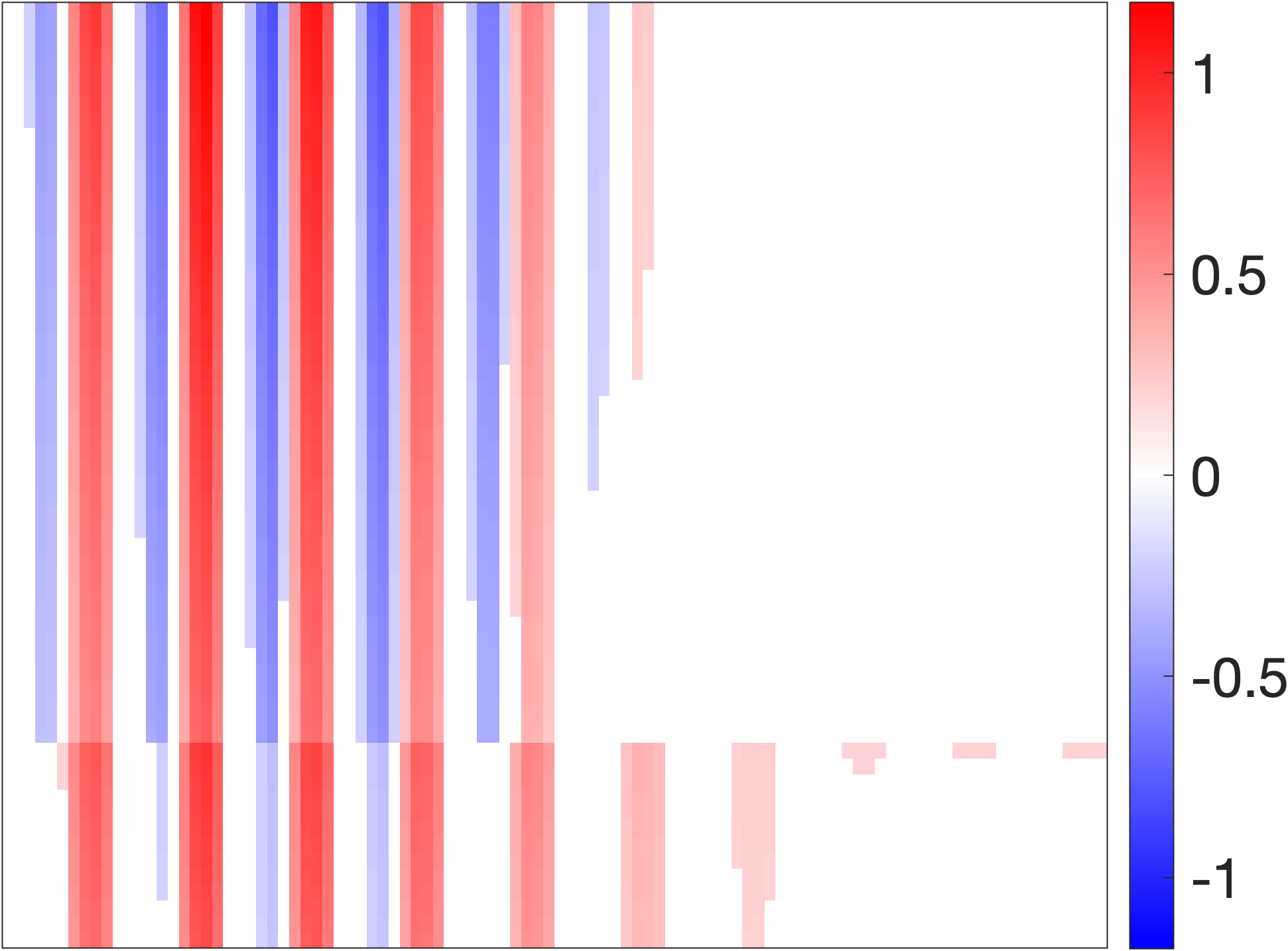}
\caption{Illustration of a \(10 \times 10\) spatial  grid for $S$ (left) and  $\widehat{S}$ (middle). Right: true coefficient matrix $\theta\in\mathbb R^{60\times100}$.
}
\label{fig: B1j}
\end{figure*}

\subsubsection{Data Generation}
\label{sec: data}
We start by detailing the data generation process, which incorporates two key characteristics: first, the $s$ spatial locations form a $\sqrt{s}\times \sqrt{s}$ grid where adjacent locations exhibit similar behavior; second, the temporal evolution is smooth, {interrupted by a single change point.}

We consider the single-response model \eqref{eq: model, single}, with noise terms generated independently as \(\varepsilon_k\overset{\mathrm{i.i.d.}}{\sim}\mathcal N(0, 10^{-4})\). The predictors $\{X_k\}_{k=1}^n$ are generated from a matrix-variate Gaussian distribution \cite{gupta1999matrix,greenewald2015robust_kronpca,tsiligkaridis2013kronecker_covariance}
as $
\mathrm{vec}(X_k)\overset{\mathrm{i.i.d.}}{\sim}\mathcal N \bigl(0,\Sigma_s\otimes\Sigma_t\bigr)
$, $k\in [n]$.
The temporal covariance {\(\Sigma_t\in\mathbb S^{t}\)} is defined by \((\Sigma_t)_{ii'}=0.9^{|i-i'|}\), while the spatial covariance is block-diagonal: \(\Sigma_s=\operatorname{diag}(\Sigma_0,\ldots,\Sigma_0) \in\mathbb S^{s}\) with \(\Sigma_0=0.4I+0.6\,\mathbf{1}\mathbf{1}^\top\in\mathbb S^{\sqrt{s}}\), where $\mathbf{1}\in\mathbb{R}^{\sqrt{s}}$ denotes the vector of all ones.
The true coefficient matrix $\theta\in \mathbb{R}^{t\times s}$ evolves over time from an initial state $\theta_{1\cdot}$ according to:
\[
\theta_{i\cdot}= 0.99\,\theta_{i-1\cdot} + 
\begin{cases}
0.2 & \text{if } i=t^\ast\\[2pt]
0 & \text{if }  i\neq t^\ast
\end{cases}, \quad \text{for}\ i = 2,\dots,t,
\]
where the change point $t^\ast$ is sampled uniformly from $\{3,\dots,t-2\}$. We construct \(\theta_{1\cdot}\) from a latent spatial matrix $S\in\mathbb{R}^{\sqrt{s}\times\sqrt{s}}$. 
The grid is  partitioned into six {roughly} equal regions ($3 \times 2$ layout). Entries in the middle-left and lower-left blocks are drawn from $\mathcal{N}(1, 0.1)$ and $\mathcal{N}(-1, 0.1)$, respectively, while the remaining four regions are set to zero (see the left panel of Fig.~\ref{fig: B1j}).
To avoid sharp spatial discontinuities, we apply Gaussian smoothing  (see Fig.~\ref{fig: B1j}, middle panel) to $S$ as
\[
\widehat{S}_{pq}= \frac{1}{2\pi} \sum_{p',q'}\exp\left({-\frac{(p-p')^2+(q-q')^2}{2}} \right){S}_{p'q'}.
\]
The initial coefficient is then defined as $\theta_{1\cdot} = \mathrm{vec}(\widehat{S})^\top$. After generating the full coefficient $\theta$ via the above temporal evolution, we threshold entries $\theta_{ij}$ with magnitude below $0.2$ to zero. The resulting true coefficient matrix $\theta$ is shown in the right panel of Fig.~\ref{fig: B1j}.

For each instance, we generate independent training, validation, and test sets with sample sizes $n$ (to be defined in the following subsections), $n_{\mathrm{val}}=10^3$, and $n_{\mathrm{test}}=10^3$, respectively.

\subsubsection{Evaluation Performance}
We compare our GGFL estimator \eqref{prob-single}, together with its simplified variant S-GGFL (which enforces $\lambda_t=\lambda_g$ in \eqref{prob-single} to reduce tuning time), against the state-of-the-art GTV estimator \cite{stevens2021graph} defined in \eqref{eq: gtv}. For all three methods, regularization parameters are selected by minimizing the validation root mean squared error (RMSE). {During the parameter tuning phase, each candidate configuration for GGFL and S-GGFL}  
is solved by applying Algorithm~\ref{alg: hpr-restart} with the stopping criterion $\eta_{\mathrm{KKT}} \le 10^{-3}$. 

To evaluate the performance of an estimator $\hat\theta$ of the true coefficient $\theta$, we report the following prediction and estimation errors:
\begin{align*}
&\mathrm{RMSE}\text{-}y
=\sqrt{ \|\mathcal X_{\rm test}(\hat\theta)-y_{\rm test} \|_2^2/n_{\rm test}},\quad
\mathrm{Error}\text{-}\theta
=\frac{\|\hat \theta-\theta\|_2}{1+\|\theta\|_2},
\end{align*}
where $\{(X_k^{\rm test},y_k^{\rm test})\}_{k=1}^{n_{\rm test}}$ denotes the test set and 
\[
\mathcal X_{\rm test}(\theta)
:=
[ \langle X^{\rm test}_1,\theta\rangle,\dots,\langle X^{\rm test}_{n_{\rm test}},\theta\rangle ]^\top.
\]

We fix dimensions $t=90,\, s=100$ and vary the training  sample size $n \in \{100, 200, 500, 1000\}$.
Regularization parameters are tuned over grids of five logarithmically spaced values:
$\lambda_t, \lambda_g \in [10^{-2}, 10^{2}]$, $\lambda_1 \in [10^{-4}, 10^{-2}]$ for GGFL/S-GGFL;
and $\lambda_1 \in [10^{-4}, 10^{-1}]$, $\lambda_{tv} \in [10^{-4}, 10^{-2}]$ for GTV.
Since GTV \eqref{eq: gtv} requires an estimated covariance matrix ${\Sigma}\in\mathbb{S}^{ts}_+$, we generate an additional $10^4$ independent samples to compute the sample covariance, applying entrywise thresholding at $0.5$ following \cite{stevens2021graph}.

\begin{table*}[htbp]
\centering
\setlength{\tabcolsep}{10pt} 
\renewcommand{\arraystretch}{1.2} 
\caption{Performance Comparison of Single-Task Estimators on Synthetic Data across Varying Training Sample Sizes $n$}
\label{tab:gtv-vs-ggfl}
\begin{tabular}{c l c c c c}
\toprule
$n$ & Estimator & RMSE-$y$ & Error-$\theta$ & Tuning Time & Final Fitting Time \\
\midrule
\multirow{3}{*}{100}
  & S-GGFL & 2.04e$+$01 & 3.12e$-$01 & 00:07:57 & 00:00:16 \\
  & GGFL   & 1.38e$+$01 & 2.12e$-$01 & 01:34:06 & 00:00:35 \\
  & GTV    & 7.70e$+$01 & 1.13e$+$02   & 02:09:07 & 00:02:34 \\
\midrule
\multirow{3}{*}{200}
  & S-GGFL & 9.17e$+$00 & 1.60e$-$01 & 00:10:00 & 00:00:18 \\
  & GGFL   & 5.88e$+$00 & 1.11e$-$01 & 00:59:40 & 00:00:27 \\
  & GTV    & 5.53e$+$01 & 1.19e$+$02   & 02:58:04 & 00:04:27 \\
\midrule
\multirow{3}{*}{500}
  & S-GGFL & 1.55e$+$00 & 5.32e$-$02 & 00:10:31 & 00:00:18 \\
  & GGFL   & 1.31e$+$00 & 4.85e$-$02 & 00:59:18 & 00:00:25 \\
  & GTV    & 1.26e$+$01 & 1.00e$+$02   & 04:24:30 & 00:01:56 \\
\midrule
\multirow{3}{*}{1000}
  & S-GGFL & 3.36e$-$01 & 2.22e$-$02 & 00:09:20 & 00:00:20 \\
  & GGFL   & 3.36e$-$01 & 2.22e$-$02 & 00:55:17 & 00:00:20 \\
  & GTV    & 6.30e$-$01 & 1.03e$+$02   & 10:02:41 & 00:06:31 \\
\bottomrule
\end{tabular}
\end{table*}

\begin{figure*}[htbp]
    \centering
    \includegraphics[width=\linewidth]{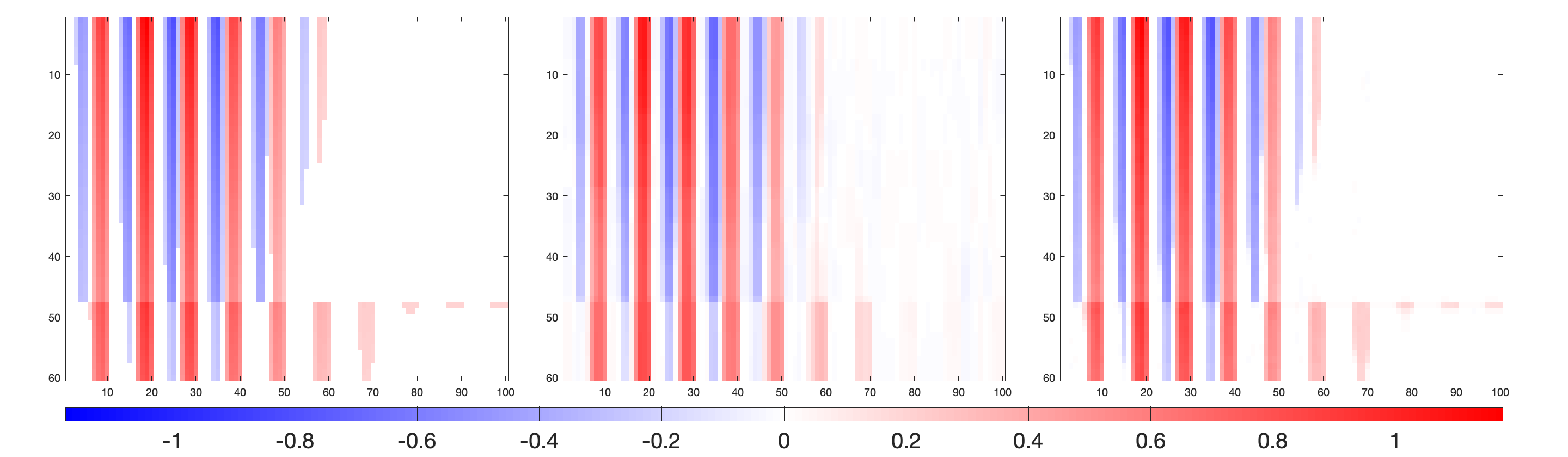}
    \caption{True coefficient matrix $\theta$ (left), and GGFL estimated coefficient matrices with sample sizes 
$n=200$ (middle) and 
$n=1000$ (right).}
    \label{fig:statistical performance single}
\end{figure*}

Table~\ref{tab:gtv-vs-ggfl} reports the test RMSE-$y$, Error-$\theta$, parameter tuning time, and final fitting time (using the {selected} optimal parameters) of three methods.
The results show that GGFL consistently achieves the lowest RMSE-$y$ and Error-$\theta$.
For example, at $n=500$, GGFL completes tuning and fitting in one hour and reaches RMSE-$y=1.31$ and Error-$\theta=4.85\times 10^{-2}$, whereas GTV requires over 4 hours, yielding {significantly} higher errors ($12.6$ and $1\times 10^{2}$).
The simplified variant S-GGFL also performs remarkably well. At $n=1000$, it requires only 10 minutes  to achieve the same {accuracy} 
as GGFL, far outperforming GTV. Notably, while GTV uses $10^4$ additional samples for covariance estimation, our estimators attain superior accuracy without this extra data, demonstrating that explicitly enforcing spatiotemporal structure offers a highly effective alternative to covariance-based modeling.

Computationally, S-GGFL and GGFL demonstrate superior efficiency and scalability.
As the sample size $n$ increases, the tuning time of GTV increases from about 2 hours to more than 10 hours, while S-GGFL consistently completes the tuning within 10 minutes and the tuning time of GGFL decreases from 2 hours to 55 minutes.
A similar trend is observed for fitting time: both GGFL and S-GGFL consistently finish within 1 minute, whereas GTV requires over 5 minutes at $n=1000$.
These results provide {strong}
evidence that our methods are both efficient and effective compared with GTV.

Furthermore, we visualize the effect of sample size on the GGFL estimator. As shown in Fig. \ref{fig:statistical performance single}, when $n=200$, the GGFL estimate shows non-negligible deviations in regions that are zero in the ground truth, reflecting limited-sample effects. In contrast, for $n=1000$, these deviations vanish and the estimated pattern closely matches the ground truth, demonstrating a substantial decrease in estimation error as the sample size grows.

Overall, S-GGFL offers a well-balanced trade-off between accuracy and efficiency, making it particularly attractive in practice.

\subsubsection{Computational Performance}
We assess the computational scalability of the HPR algorithm (Algorithm~\ref{alg: hpr-restart}) for solving \eqref{prob-single} as the temporal length $t$ and spatial dimension $s$ increase. We further evaluate its scalability in the multi-task setting \eqref{prob-multi} by varying the number of tasks $m$, along with $t$ and $s$. For the multi-task experiments, data are generated from the multi-task model \eqref{eq: model, multi}, where {the coefficient matrix $\Theta^{(r)}$ for} each task is constructed as Section \ref{sec: data} {with all tasks sharing}
a common {temporal} change point $t^*$.

We first evaluate the scalability of the HPR algorithm for solving the single-task GGFL model with sample size $n=500$ by varying the temporal length $t$ and spatial dimension $s$. Specifically, we consider $t \in [60, 240]$ with $s=100$ fixed, and $s \in [100, 400]$ with $t=60$ fixed. The results are shown in Fig.~\ref{fig: single-task scalability}. We then conduct analogous experiments for the multi-task setting (MultiGGFL) with $m=5$, varying $t$ and $s$ under the same settings, see the upper and middle panels of Fig.~\ref{fig: multi-task scalability}. We further assess the scalability with respect to the number of tasks by increasing $m \in [5, 25]$ while fixing $n=500, t=60, s=100$ (Fig.~\ref{fig: multi-task scalability}, lower panel). In all experiments, we test four regularization strengths $\lambda_0 \in \{10^{-3}, 10^{-2}, 10^{-1}, 1\}$, setting all penalty parameters equal to $\lambda_0$ (i.e., $\lambda_{i} = \lambda_0$, $i=1,2,t,g$).

\begin{figure}[htbp]
    \centering
\includegraphics[width=0.45\linewidth]{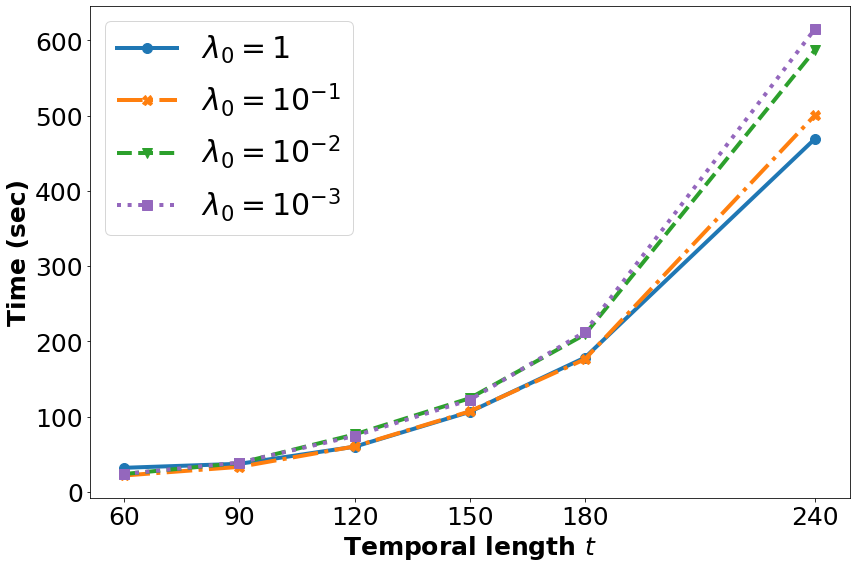}
\hspace{0.35cm}
\includegraphics[width=0.45\linewidth]{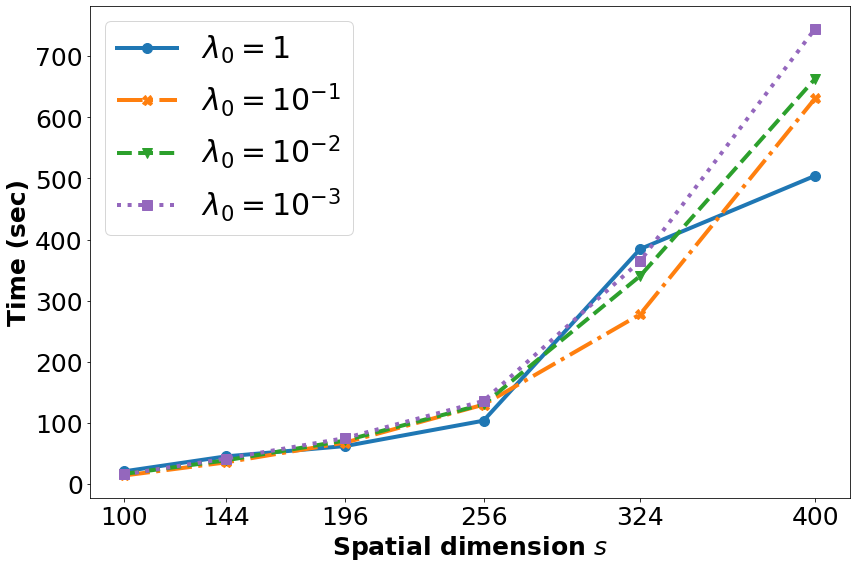}
\caption{Runtime of HPR  for GGFL  under varying regularization strengths $\lambda_0$. {Upper:} Scaling with temporal length $t$ (fixed $n=500, s=100$). {Lower:} Scaling with spatial dimension $s$ (fixed $n=500, t=60$). }
    \label{fig: single-task scalability}
\end{figure}

\begin{figure}[htbp]
    \centering
\includegraphics[width=0.32\linewidth]{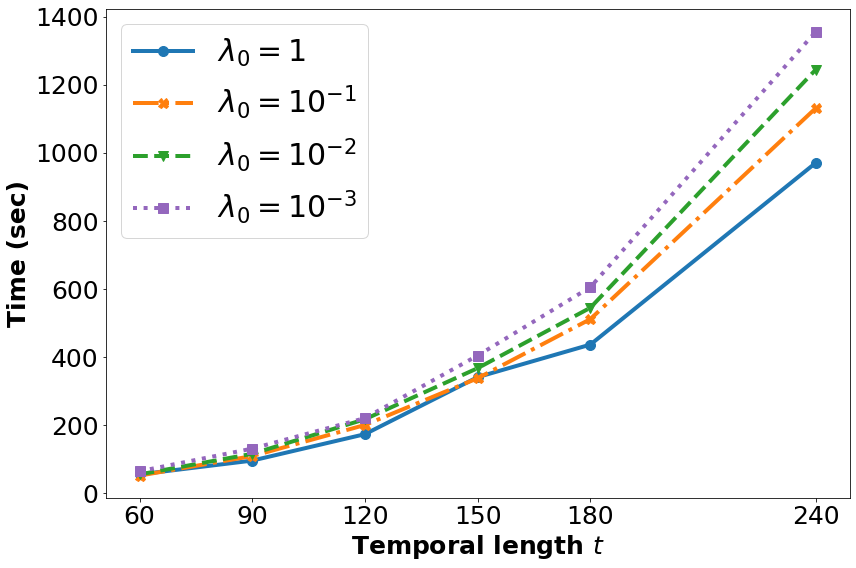}
% \hspace{0.35cm}
\includegraphics[width=0.32\linewidth]{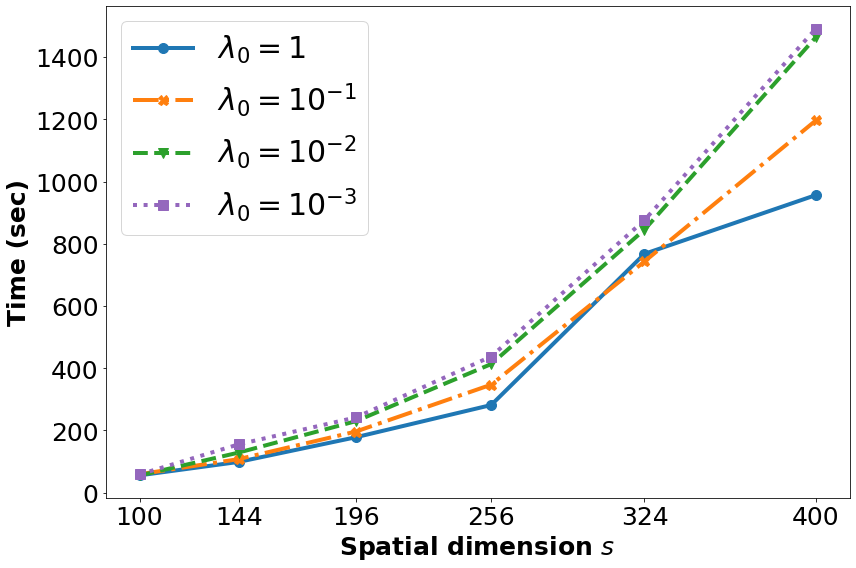}
% \hspace{0.2cm}
\includegraphics[width=0.32\linewidth]{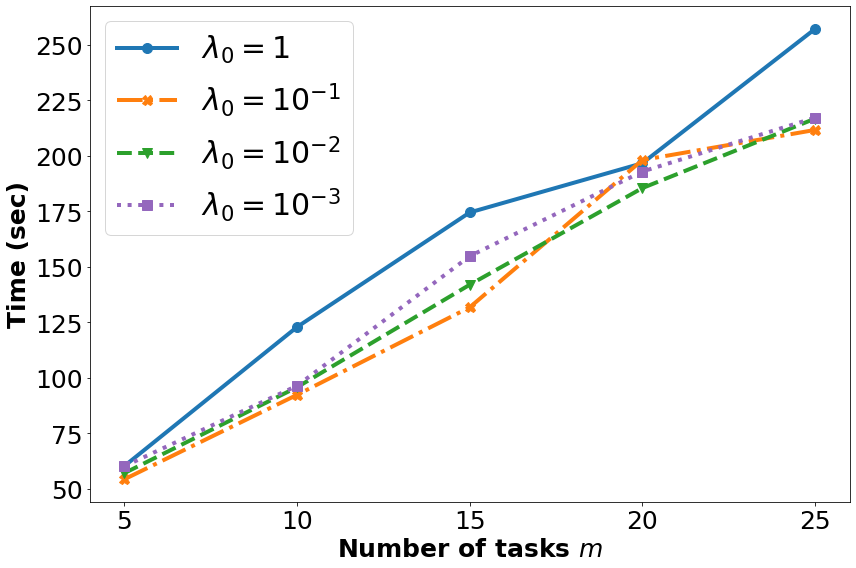}
\caption{Runtime of HPR for MultiGGFL under varying regularization strengths $\lambda_0$. Upper: Scaling with temporal length $t$ (fixed $n=500, s=100, m=5$). Middle: Scaling with spatial dimension $s$ (fixed $n=500, t=60, m=5$). Lower: Scaling with number of tasks $m$ (fixed $n=500, t=60, s=100$).} 
    \label{fig: multi-task scalability}
\end{figure}

We can see from the figures that the computational time scales approximately linearly with $m$ and at most quadratically with $t$ and $s$, consistent with the complexity analysis in Section~\ref{sec: updata_complexity}.
Notably, the performance is generally insensitive to regularization strength, with negligible runtime differences across different $\lambda_0$.
These findings confirm that the HPR algorithm is robust and stable   across varying problem dimensions and regularization strengths.

\subsection{SWUS Precipitation Forecasting: Single-Task}

\begin{figure}[htbp]
\centering
{\includegraphics[width=0.3\linewidth]{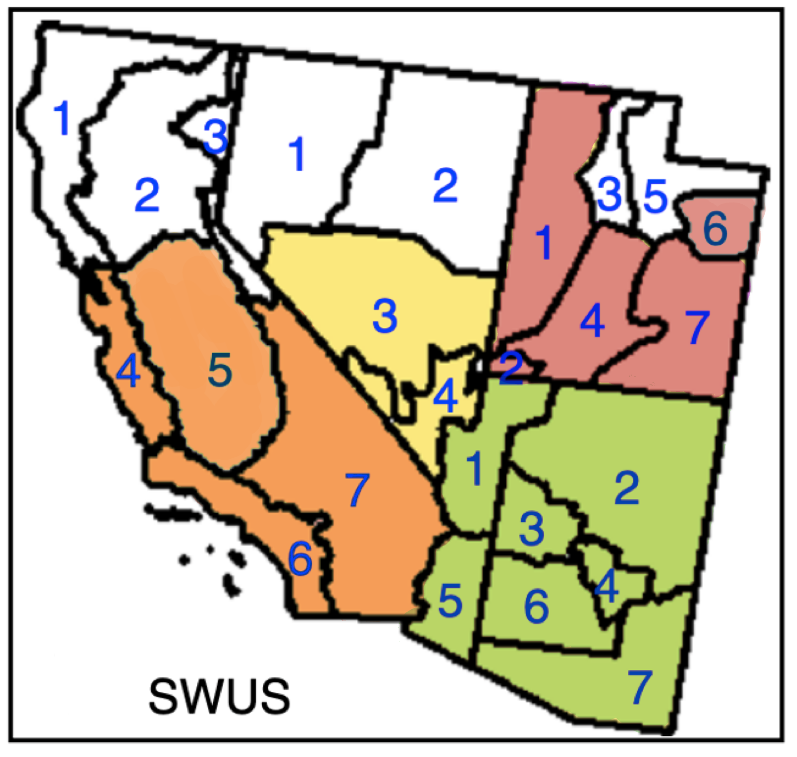}}
\quad
{\includegraphics[width=0.3\linewidth]{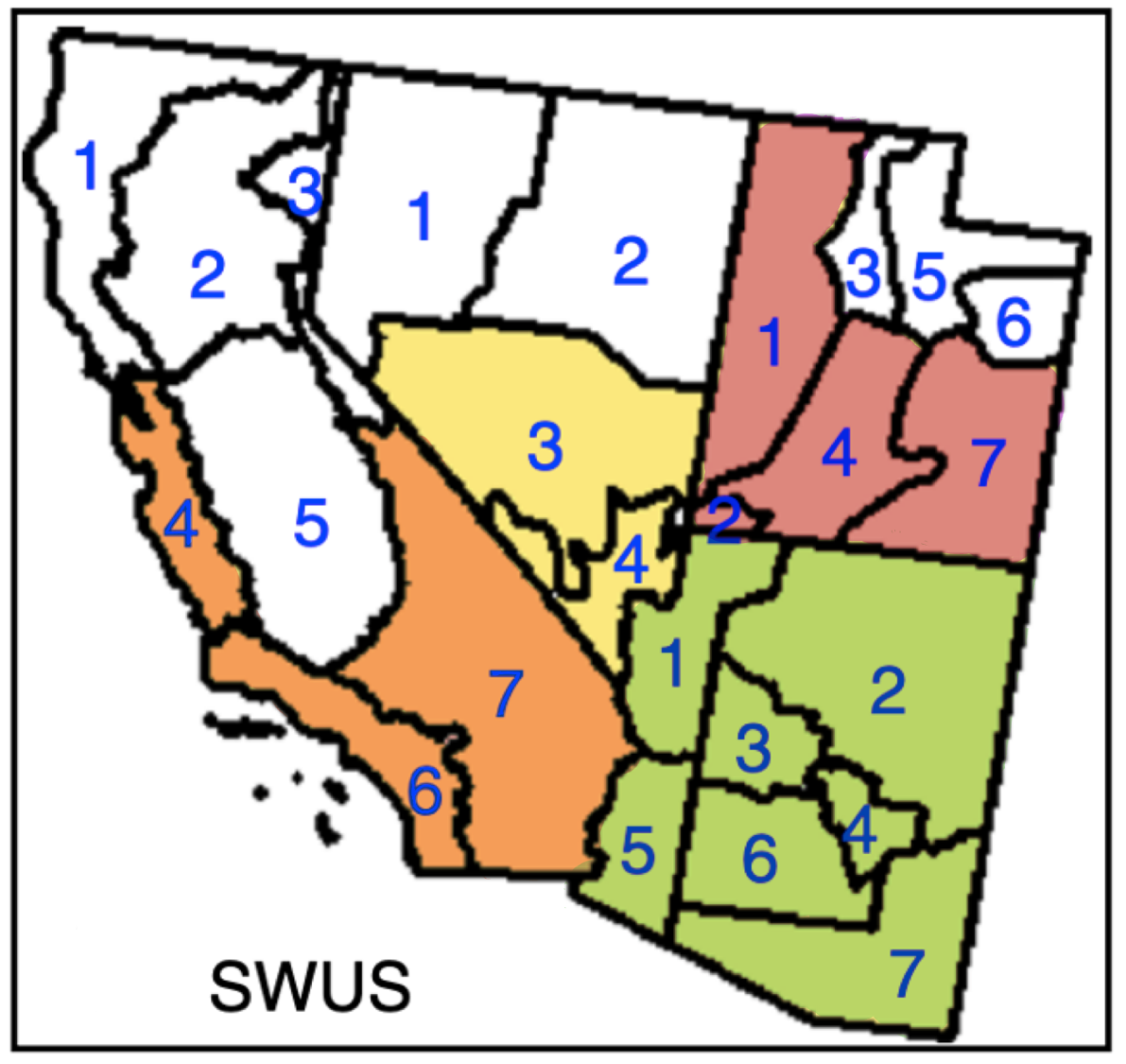}}
\caption{Illustration of selected SWUS climate divisions. 
Left: Divisions used to compute the area‐weighted average winter precipitation response.  
Right: The 16 distinct  divisions, each corresponding to a separate task in the multi-task setting.
Regions are color‐coded by state: California (orange), Nevada (yellow), Utah (red), and Arizona (green). 
}
\label{fig:swus}
\end{figure}

We evaluate the GGFL \eqref{prob-single} and GTV \eqref{eq: gtv} models {for} forecasting winter (November--March) precipitation in the southwestern United States (SWUS). Following the setup in \cite{stevens2021graph}, the dataset covers the period 1940--2018.
For each year $k$, the response $y_k$ is the area-weighted average winter precipitation over selected SWUS climate divisions ({see} Fig. \ref{fig:swus}, left panel).
The predictor $X_k \in \mathbb{R}^{t \times s}$ consists of sea surface temperatures (SST) from the preceding $t=4$ months (July--October) across $s=220$ spatial locations on a $10^\circ\times10^\circ$ grid over the Pacific ocean.
In our experiment, we use preprocessed data  from the GTV package\footnote{\url{https://github.com/Willett-Group/gtv_forecasting}}.

We partition the data into training (1940--1989, $n=50$) and testing (1990--2018, $n_{\mathrm{test}}=29$) sets. During training, we replace the {standard squared $\ell_2$-loss}
in \eqref{prob-single} and \eqref{eq: gtv} by a temporally weighted {least squares}
to prioritize recent observations:
\begin{equation}\label{eq:timeweight}
 \bigl\|D^{\frac{1}{2}}\bigl(y-\mathcal{X}(\theta)\bigr)\bigr\|_2^2 , \mbox{ with }
D=\mathrm{diag}\bigl(\{0.9^{\,1989-k}\}_{k=1940}^{1989}\bigr).
\end{equation}

Parameters are selected via five-fold cross-validation on the training set. We tune $\lambda_1 \in [10^{-4}, 10^{-2}]$ and $\lambda_{t}=\lambda_{g} \in [10^{-3}, 10^{2}]$ for GGFL, and $\lambda_1 \in [10^{-4}, 10^{-1}]$ and $\lambda_{tv} \in [10^{-4}, 10^{-2}]$ for GTV, using 20 logarithmically spaced values for each parameter. This parameter tuning for GTV covers a broader and finer range than that used in the GTV package \cite{stevens2021graph}.

Since GTV requires a covariance estimate ${\Sigma} \in \mathbb{S}^{ts}$, we evaluate two variants:
 {GTV(Obs)}, using the sample covariance computed directly from the training predictors; and
 {GTV(LENS)}, using external data from the CESM Large Ensemble dataset \cite{kay2015community} following \cite{stevens2021graph}.
Both variants apply a threshold of $0.5$ to ${\Sigma}$ to control graph sparsity.

\begin{table}[H]
\centering
\caption{Forecasting Comparison for SWUS Winter Precipitation
}
\label{tab:swus_single_task}
\setlength{\tabcolsep}{10pt}
\renewcommand{\arraystretch}{1.2}
\begin{tabular}{l c c c}
\toprule
Estimator & Test RMSE-$y$ & Tuning Time & Fitting Time \\
\midrule
GGFL            & 0.904 & 00:12:00 & 00:00:01 \\
GTV(LENS) & 0.983 & 00:43:49 & 00:00:02 \\
GTV(Obs)       & 1.076 & 01:00:24 & 00:00:02 \\
\bottomrule
\end{tabular}
\end{table}

Table~\ref{tab:swus_single_task} summarizes the  forecasting results.  As can be seen, GGFL achieves the lowest test RMSE-$y$ (0.904),  outperforming GTV(LENS) (0.983) and GTV(Obs) (1.076).
{We note that \cite{stevens2021graph} reported a lower RMSE-$y$ of 0.860 for GTV(LENS), achieved via {a specific selection of hyperparameter candidates.} 
This discrepancy likely stems from the sensitivity of GTV to hyperparameter choices in limited-data regimes.}
Consistent with \cite{stevens2021graph}, GTV(LENS) outperforms GTV(Obs), confirming that incorporating external  data improves covariance estimation.
Computationally, GGFL is significantly more efficient, completing tuning in 12 minutes compared to 44 minutes for GTV(LENS) and 1 hour for GTV(Obs). 
These results show that GGFL consistently outperforms GTV in both predictive accuracy and computational efficiency.

\subsection{SWUS Precipitation Forecasting: Multi-Task}
{For the multi-task experiments, w}e evaluate the performance under both joint (multi-task) and independent (single-task) settings.
We compare MultiGGFL \eqref{prob-multi} against MultiGTV \cite[(3)]{stevens2019graph}, which minimizes:
\begin{align*}
&\min_{\Theta\in \mathbb{R}^{t\times s\times m}} \ \sum_{r=1}^m\Bigl[
\|y^{(r)}-\mathcal{X}(\Theta^{(r)})\|_2^2
+ \lambda_1 \|\Theta^{(r)}\|_1 + \lambda_{1} \sum_{(i,j),(i',j')} |{\Sigma}_{(i,j),(i',j')} |^{\frac{1}{2}} \big|\Theta_{ij}^{(r)} - {s}_{(ij),(i'j')}\Theta^{(r)}_{i'j'}\big| \Bigr] \\
&\quad\quad\quad\quad\quad \quad+ \lambda_2 \sum_{i=1}^t \sum_{j=1}^s \|\Theta_{[ij]}\|_2.
\end{align*}
In the single-task setting, GGFL \eqref{prob-single} and GTV \eqref{eq: gtv} are fitted independently to each {task.}

{In this setup, w}e define $m=16$ tasks corresponding to 16 SWUS climate divisions shown in the right penal of Fig.~\ref{fig:swus}.
The response $y_k^{(r)}$ is the winter precipitation for division $r$, sourced from NOAA\footnote{\url{https://www.ncei.noaa.gov/pub/data/cirs/climdiv/} (dataset \texttt{climdiv-pcpndv-v1.0.0-20250905})}, as divisional precipitation data are not included in the GTV package. Predictors $X_k$ remain identical to the single-task setting.
{Data partitioning and the temporal weighting scheme of the loss also follow the protocol described in the previous subsection.}

Parameters are tuned via five-fold cross-validation.
For MultiGGFL, we select $\lambda_1=\lambda_2$ from 20 logarithmically spaced values in $[10^{-4},10^{-2}]$, and $\lambda_t=\lambda_g$ from 20 values in $[10^{-3},10^{2}]$.
For MultiGTV, $\lambda_1$ and $\lambda_2$ are tuned over 10 logarithmically spaced values in $[10^{-3},10^{2}]$. Due to the high computational cost of MultiGTV, we restrict the grid size and cap the runtime at 1800 seconds per instance.

\begin{figure*}[htbp]
\centering
\includegraphics[width=0.95\linewidth]{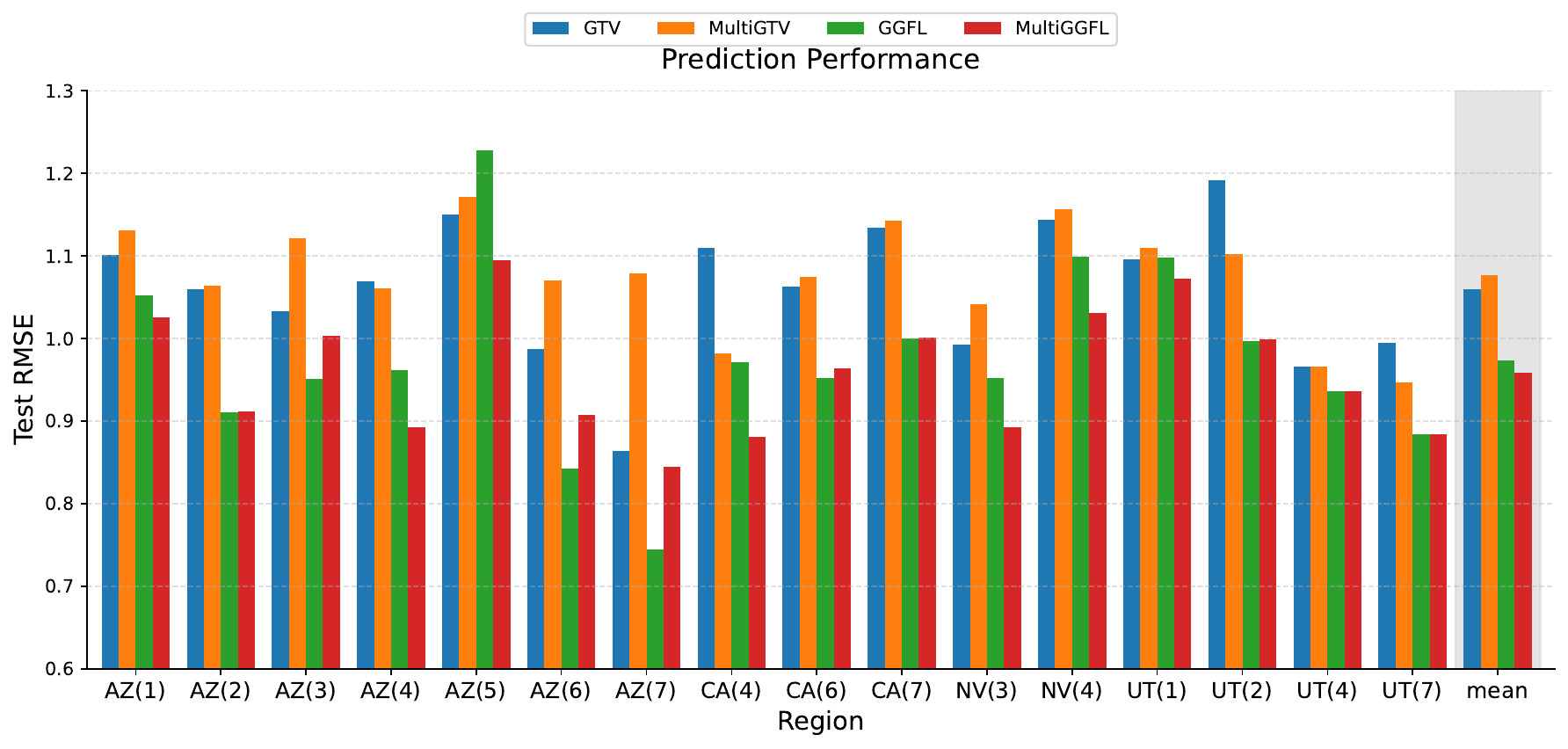}
\caption{Forecasting comparison for SWUS winter precipitation across 16 divisions. The rightmost bar denotes the mean across all divisions. 
Computational times (Tuning / Fitting): GTV (10 hours / 2 minutes), MultiGTV (72 hours / 30 minutes), GGFL (3 hours / 20 seconds), and MultiGGFL (3 hours / 14 seconds).
}
\label{fig:real_data_multitask_obs}
\end{figure*}

Fig.~\ref{fig:real_data_multitask_obs} reports test RMSE-$y$ by division for all four models. 
The two GGFL variants achieve consistently lower test errors than their GTV counterparts, both across most divisions and on average. This highlights the advantage of  explicitly enforcing temporal smoothness across consecutive time points and spatial similarity among adjacent regions.
Although GGFL performs competitively overall, it exhibits higher variance across divisions. For instance, it achieves the best results in AZ(6) and AZ(7), but performs the worst in AZ(5). In contrast,  MultiGGFL attains the lowest mean RMSE-$y$ and exhibits smaller variability, showing that jointly learning across related tasks enhances both stability and accuracy.
MultiGTV also improves upon GTV via task coupling but remains inferior to MultiGGFL.
Computationally, both GGFL and MultiGGFL complete tuning and fitting within 3 hours. In contrast, GTV requires nearly 10 hours, and MultiGTV spans nearly 3 days.
These results demonstrate that the proposed framework provides robust and efficient precipitation forecasting, with the multi-task formulation offering the greatest overall gains.

\section{Conclusion}
\label{sec: conclusion}
We present a unified framework for spatiotemporal matrix regression in single- and multi-task settings. Our GGFL {estimator leverages} 
a weighted spatial graph to couple neighboring locations while simultaneously {enforces}
temporal smoothness and {feature} sparsity{. Building on this,} 
the multi-task extension further aligns structure across related tasks. {On the theoretical front, we establish $\sqrt{n}$-consistency for our estimators.}  
On the computational side, we employ the HPR algorithm for solving the resulting {composite convex}
problem with a non-ergodic $\mathcal{O}(1/k)$ rate and low per-iteration complexity, making the approach practical at scale.  
{Numerical results indicate that} 
the proposed methods show superior predictive performance and estimation accuracy compared with the state-of-the-art GTV estimators,  while maintaining strong efficiency and scalability. {These findings validate that} 
our estimators offer a robust and interpretable route to high-dimensional spatiotemporal learning.

{Several directions need further investigation. First, our framework could incorporate} 
robust or generalized losses {(e.g., Huber, logistic, or Poisson), allowing for the handling of} 
heavy tails, outliers, {and categorical data for classification.}
{Second, rather than pre-specifying structures, adaptively learning spatiotemporal graphs from data \cite{lin2024dnnlasso} presents a promising avenue. Finally, exploring second-order information may further accelerate computation, offering a more efficient alternative for large-scale problems.}

\bibliographystyle{ieeetr}
\bibliography{references}
\end{document}